\documentclass[11pt, reqno]{amsart}%
\usepackage{amssymb}
\usepackage{amsmath,esint}
\usepackage[shortlabels]{enumitem}
\setlist[enumerate]{leftmargin=*} 
\setlist[itemize]{leftmargin=*, align=parleft,left=0pt..1em}
\usepackage{amsthm, dsfont}
\usepackage{amscd,mdwlist}
\usepackage{color}
\usepackage{cite}
\usepackage{bbm}
\usepackage{tcolorbox}
\usepackage{thmtools}
\usepackage[backgroundcolor=white, bordercolor=blue,
linecolor=blue]{todonotes}
\usepackage{relsize}
\usepackage[pagebackref=false, pdfpagelabels, plainpages=false]{hyperref}

\setlist[description]{leftmargin=\parindent,labelindent=\parindent}
\setlist[enumerate]{leftmargin=8mm ,labelindent=0cm}

\hypersetup{
colorlinks=false,
linkcolor=red,
filecolor=magenta,
urlcolor=cyan,
citecolor=blue,
}

\usepackage[
paper=a4paper,
left=2cm,
right=25mm,
top=2.5cm,
bottom=2cm,
includefoot,
foot=8mm,
bindingoffset=0mm]{geometry}

\theoremstyle{plain}
\declaretheorem[title=Theorem, parent=section]{theorem}
\declaretheorem[title=Lemma,sibling=theorem]{lemma}
\declaretheorem[title=Proposition,sibling=theorem]{proposition}
\declaretheorem[title=Corollary,sibling=theorem]{corollary}

\theoremstyle{definition}
\declaretheorem[title=Definition,sibling=theorem]{definition}
\declaretheorem[title=Remark,sibling=theorem]{remark}
\declaretheorem[title=Remark, numbered=no]{remark*}
\declaretheorem[title=Example, sibling=theorem]{example}

\numberwithin{equation}{section}

\DeclareMathOperator*{\esssup}{ess\,sup}
\DeclareMathOperator*{\essinf}{ess\,inf}

\newcommand{\il}{\int\limits}
\newcommand{\iil}{\iint\limits}
\DeclareMathOperator{\R}{\mathbb{R}}
\renewcommand{\d}{\,\mathrm{d}}

\newcommand{\WnuOm}{W_{\nu}^{p}(\Omega)}

\newcommand{\eps}{\varepsilon}
\newcommand{\uu}{\mathring{\text{u}}}

\newcommand{\vertiii}[1]{{\left\vert\kern-0.25ex\left\vert\kern-0.25ex\left\vert #1 \right\vert\kern-0.25ex\right\vert\kern-0.25ex\right\vert}}

\begin{document}
\title{Nonlocal Gagliardo-Nirenberg-Sobolev type inequality \\
\url{https://dx.doi.org/10.4310/CMS.241217014531}
}

\author{Guy Foghem}
\address{{\tiny Technische Universt\"{a}t Dresden, Fakult\"{a}t f\"{u}r Mathematik, Institut f\"{u}r Wissenschaftliches Rechnen, Willers-Bau, B 220, Zellescher Weg 12-14 01069, Dresden, Germany. Email: guy.foghem@tu-dresden.de}}

\thanks{Financial support by the DFG via the Research Group 3013: ``Vector-and Tensor-Valued Surface PDEs'' is gratefully acknowledged.}

\begin{abstract}
We establish Gagliardo-Nirenberg-Sobolev type inequalities on nonlocal Sobolev spaces driven by $p$-L\'{e}vy  integrable kernels, by imposing some appropriate growth conditions on the associated critical function.        
This naturally allows to devise Sobolev embeddings, as well as, compact embeddings of nonlocal Sobolev spaces into Orlicz type spaces. 
The Gagliardo-Nirenberg-Sobolev type inequalities, as in the classical context, turn out to have some reciprocity with Poincar\'{e} and Poincar\'{e}-Sobolev type inequalities. 
The classical fractional Sobolev inequality is also derived as a direct consequence. 
\end{abstract}



\maketitle

\section{Introduction}
\subsection{Motivation} Classical Sobolev inequalities are ubiquitous within the area of partial differential equations and calculus of variations, and have been investigated by a numerous number of mathematicians. They play crucial roles in existence theory and regularity theory. The Gagliardo-Nirenberg-Sobolev inequality, amongst many others, is certainly the most significant and influential Sobolev inequality. The aim of this work to establish analog of such an inequality on nonlocal Sobolev spaces generated by $p$-L\'{e}vy  integrable kernels, that can be seen as a genuine generalization of fractional Sobolev-Slobodeckij spaces. 
Our exposition aims to be as self-contained as possible. 
To begin with, let us introduce nonlocal Sobolev space with respect to a $p$-L\'{e}vy integrable kernel $\nu$. Throughout this work the kernel $\nu: \R^d\setminus\{0\}\to [0,\infty)$, $d\geq1$, is assumed to be the density of a symmetric $p$-L\'{e}vy  measure with $1\leq p<\infty$ that is $\nu$ is symmetric, i.e., $\nu(h) = \nu(-h)$ for $h\in \R^d\setminus\{0\}$ and $\nu$ is $p$-L\'{e}vy  integrable, i.e., 
\begin{align}\label{eq:plevy-integrable}
\int_{\R^d} (1\land|h|^p)\nu(h)\d h<\infty.
\end{align}

\noindent Hereafter, we write $|h|= (h_1^2+ h_2^2+\cdots+ h_d^2)^{1/2}$ and $a\land b= \min(a,b)$ for $a,b\in \R$. The associated nonlocal Sobolev space  is $W_\nu^p(\R^d)= \big\{u\in L^p(\R^d):~ |u|_{W_\nu^p(\R^d)}<\infty \big\}$ where 
\begin{align}\label{eq:nonsemi-norm}
&|u|_{W_\nu^p(\R^d)}= \Big(\iil_{\R^d\R^d} |u(x)-u(y)|^p\nu(x-y)\d y\d x\Big)^{1/p}. 
\end{align}
The space $W_\nu^p(\R^d)$ amounts to a Banach space under the norm
\begin{align*}
\|u\|_{W_\nu^p(\R^d)}= \big(\|u\|^p_{L^p(\R^d)}+ |u|^p_{W_\nu^p(\R^d)}\big)^{1/p}. 
\end{align*}

\noindent The terminology \emph{nonlocal Sobolev space} to designate the space $W_\nu^p(\R^d)$ is justified since the latter naturally arises as the energy space associated with a nonlocal operator, which is a (non)linear $p$-L\'{e}vy  integro-differential operator generated by $\nu$, of the form 
\begin{align*}
Lu (x):= 2\operatorname{p.v.}\int_{\R^d} |u(x)-u(y)|^{p-2}(u(x)-u(y))\nu(x-y)\d y,\qquad\qquad (x\in \R^d).
\end{align*} 
Indeed the $p$-L\'{e}vy operator $Lu$  appears as the subdifferential of  $|u|^p_{W^p_\nu(\R^d)}$ and, moreover, it is readily seen that $\langle Lu,u\rangle=|u|^p_{W^p_\nu(\R^d)}$ for a sufficiently smooth function $u$. We refer interested reader to \cite{FK22,FPS23,Fog23s,DFK22} recent studies on Integro-Differential Equations (IDEs) associated with the $p$-L\'{e}vy  operator $L$.  
We emphasize that  for $p=2$ the condition \eqref{eq:plevy-integrable} is known in the theory of probability analysis as the L\'{e}vy condition  and in this case the operator $ L$ naturally arises as  the generator of a L\'{e}vy stochastic processes with pure jumps; see \cite{Sat13} for more.  Therefore terming the condition \eqref{eq:plevy-integrable} naturally as the $p$-L\'{e}vy  integrability condition is arguably much more congruent.  It is worth noting  that the $p$-L\'{e}vy  integrability condition \eqref{eq:plevy-integrable} is not fortuitous and is intrinsically related to the space $W^p_\nu(\R^d)$. In fact, the $p$-L\'{e}vy  integrability condition \eqref{eq:plevy-integrable} can be self-generated from the seminorm $|\cdot|_{W^p_\nu(\R^d)}$ in the sense that, $\nu$ is  $p$-L\'{e}vy  integrable if and only if $|u|_{W^p_\nu(\R^d)}<\infty$ for all $u\in C_c^\infty(\R^d)$.  Further,  $p$-L\'{e}vy  integrable if and only if $W^p_\nu(\R^d)\neq \{0\}$ if and only if the inclusion $W^{1,p}(\R^d)\subset  W_\nu^p(\R^d)$ holds and is continuous; see Theorem  \ref{thm:charac-plevy-radial-bis}. Here $W^{1,p}(\R^d)$ is the classical Sobolev space, i.e.,  the space of functions in $ L^p(\R^d)$ whose first order distributional derivatives also lie in $ L^p(\R^d)$. 
Recent studies regarding the function space $W_\nu^p(\R^d)$ and the analysis of related integro-differential equations  can be found in \cite{Fog23s,FK22, Fog20}.
\noindent The class of $p$-L\'{e}vy  integrable kernels includes not only integrable functions, but also kernels with heavy singularities at the origin. A prototypical example is given by  $\nu(h) =s(1-s) |h|^{-d-sp}$,  whereby it is readily seen that $\nu\in L^1(\R^d,1\land|h|^p)$ if and only if $s\in (0,1)$. The resulting space is thus the well-known fractional Sobolev-Slobodeckij space $W^{s,p}(\R^d)$ of order $s$. Just like the latter, the nonlocal Sobolev space $W_\nu^p(\R^d)$ also appears as a refinement space between $L^p(\R^d)$ and the classical Sobolev space $W^{1,p}(\R^d)$.   Furthermore, as $s\to1^-$, the space $W^{s,p}(\R^d)$ reduces to the classical Sobolev space $W^{1,p}(\R^d)$. Rigorously speaking, if $|u|_{W^{s,p}(\R^d)}$ is given by \eqref{eq:nonsemi-norm} for $\nu(h)= s(1-s)|h|^{-d-sp}$, $\nabla u$ denotes the distributional gradient of $u$ and  
$|u|_{W^{1,p}(\R^d)}=\|\nabla u\|_{L^{p}(\R^d)}$ is the $L^p$-norm of $|\nabla u|$
then asymptotically we have 
\begin{align}\label{eq:frac-asypp-BBM}
\lim_{s\to 1^-}|u|^p_{W^{s,p}(\R^d)}= \frac{|\mathbb{S}^{d-1}|}{p} K_{d,p}|u|^p_{W^{1,p}(\R^d)}. 
\end{align}
Here $K_{d,p}$ is a universal constant, \cite{Fog23}, given for any unit vector $e\in \mathbb{S}^{d-1}$ by
\begin{align*}
K_{d,p}= \fint_{\mathbb{S}^{d-1}} |w\cdot e|^p\sigma_{d-1}(w)=\frac{\Gamma\big(\frac{d}{2}\big)\Gamma\big(\frac{p+1}{2}\big)}{\Gamma\big(\frac{d+p}{2}\big) \Gamma\big(\frac{1}{2}\big)}. 
\end{align*}
\noindent Another side motivation to studying the class of nonlocal Sobolev spaces under consideration is that, the asymptotic convergence in \eqref{eq:frac-asypp-BBM} remains true for a general family $(\nu_\eps)_{\eps>0}$ of radial $p$-L\'{e}vy  integrable kernels $\nu_\eps : \R^d\setminus\{0\}\to [0,\infty)$ satisfying 
\begin{align}\label{eq:approx-dirac-levy}
\int_{\R^d}(1\land|h|^p) \nu_\eps (h)\d h=1\quad\text{and for all $\delta>0,$}\quad \lim_{\eps\to 0^+} \int_{|h|>\delta}(1\land|h|^p) \nu_\eps (h)\d h=0. 
\end{align}

\noindent As a result, see \cite{Fog23s}, the conditions in \eqref{eq:approx-dirac-levy} hold if and only if  for all $u\in W^{1,p}(\R^d)$, 
\begin{align}\label{eq:gene-asypp-BBM}
\lim_{\eps\to 0^+}\iil_{\R^d\R^d} |u(x)-u(y)|^p\nu_\eps(x-y) \d y\d x= K_{d,p}|u|^p_{W^{1,p}(\R^d)}.
\end{align}
We also refer to \cite{Fog23,Fog20,BBM01, FGKV20} where the asymptotic formula \eqref{eq:gene-asypp-BBM} is established. A remarkable example of sequence $(\nu_\varepsilon)_\varepsilon$ generated after rescaling of a radial $p$-L\'{e}vy  integrable function $\nu: \R^d\setminus\{0\}\to [0, \infty)$ and satisfying the properties in \eqref{eq:approx-dirac-levy} is defined as follows 
\begin{align*}
\begin{split}
\nu_\varepsilon(h) = 
\begin{cases}
\varepsilon^{-d-p}\nu\big(h/\varepsilon\big)& \text{if }~~|h|\leq \varepsilon\\
\varepsilon^{-d}|h|^{-p}\nu\big(h/\varepsilon\big)& \text{if }~~\varepsilon<|h|\leq 1\\
\varepsilon^{-d}\nu\big(h/\varepsilon\big)& \text{if }~~|h|>1,
\end{cases}\,\,\,\text{provided}\quad \int_{\R^d}(1\land|h|^p)\,\nu(h)\d h=1.
\end{split}
\end{align*}
\vspace{-1mm}
The simplest version of Gagliardo-Nirenberg-Sobolev inequality, reads as follows; for $0\leq s\leq 1$ and $1\leq p<\infty$, if the critical Sobolev exponent $p^*_s$ also called the Sobolev conjugate of $p$ satisfies 
$$\frac{1}{p^*_s}:= \frac{1}{p}- \frac{s}{d}>0,$$
then there exists a constant $C_s= C(d,s,p)>0$ such that 
\begin{align}\label{eq:frac-sobolev-inequality}
\Big(\int_{\R^d} |u(x)|^{p^*_s}\d x\Big)^{1/p_s^*} \leq C_s|u|_{W^{s,p}(\R^d)}\qquad\text{for all $u\in L^{p^*_s}(\R^d)$}. 
\end{align}
\noindent The inequality \eqref{eq:frac-sobolev-inequality} is tautological for $s=0$, with the convention $W^{0,p}(\R^d)= L^p(\R^d)$. The proof for $s=1$ can be found in classical books, e.g., \cite{Adams,Zi12, Sal02}.
Historically the inequality \eqref{eq:frac-sobolev-inequality} is named as the Gagliardo-Nirenberg-Sobolev inequality because the case $s=1$ and $1<p<d$ is due to Sobolev \cite{Sob38} whose original proof relies upon the boundedness of the Riesz potential from $L^p(\R^d)$ to $L^{p^*_1}(\R^d)$ whereas, Gagliardo \cite{Gag59} and Nirenberg \cite{Nir59} independently provided the complete picture of inequality \eqref{eq:frac-sobolev-inequality} when $s=1$ including the case $p=1$ much later in a more general context of the well-known \textit{Gagliardo-Nirenberg interpolation inequality}; see \cite{FFMR21} for recent progress on this topic. For $ s\in (0,1)$,  the direct proof that we present in Theorem \ref{thm:gagliardo-sobolev}, for the convenience of the reader, is apparently due to Haim Brezis and can be found in \cite[Proposition 15.5]{Ponce16elliptic}. It is important to highlight that earlier proofs of the inequality \eqref{eq:frac-sobolev-inequality} exist in the literature as well. For instance, a proof using basic analysis tools is well incorporated in \cite[Section 6]{Hitchhiker} which originally springs from \cite{SV11}. See also \cite{BBM02,MS02} where the fractional inequality is established with a robust constant, i.e., with a constant $C_s$ that stays asymptotically equivalent to $(1-s)^{1/p}$ as $s\to 1^-$. The best constant of the inequality \eqref{eq:frac-sobolev-inequality}, when $s=1$, is exhibited in the case $p>1$ in \cite{Tal76}  and in the case $p=1$ in \cite[Section 6]{FF60}; see also \cite[Theorem 2.7.4]{Zi12} or \cite{Sal02} for more detailed proofs. The best constant of the fractional Gagliardo-Nirenberg-Sobolev inequality  in the special case $p=2$ and $ s\in (0,1)$ is established in \cite{CT04}. Related Sobolev type inequalities with best constants such as log-Sobolev inequality, Hardy-Sobolev inequality or Gagliardo-Nirenberg inequality are referenced in \cite{FS08,HT22}. 
\subsection{Main goal}
In view of the aforementioned classical (fractional) Gagliardo-Nirenberg-Sobolev inequality \eqref{eq:frac-sobolev-inequality}, it is reasonable to seek for an analogous inequality for the nonlocal Sobolev space $W_\nu^p(\R^d)$. Accordingly, we need to enforce adequate assumptions regarding the $p$-L\'{e}vy  kernel $\nu$. 
First and foremost, for $r>0$, consider the \textit{rearrangement radius function} $\eta$ defined by 
\begin{align*}
\text{$\eta(r) = (\frac{r}{c_d})^{1/d}$ \,\, with \,\, $c_d=|B(0,1)|$}.
\end{align*} 
Indeed, for every measurable set $E\subset \R^d$,  there holds that  $|E|=|B(0,\eta(|E|))|$.  Next we need  introduce the potential $w:[0,\infty)\to [0,\infty)$ defined by 
\begin{align}\label{eq:defw1}
w(r) = \Big(|B(0,\eta(r))| \int_{B^c(0, \eta(r))} \nu(h) \d h \Big)^{1/p} 
\quad\text{equally}\quad 
\frac{w^p(r)}{r} = \int_{B^c(0, \eta(r))} \nu(h) \d h.
\end{align}
The potential $w$, which is essential in our analysis, avoids the eventual singularity of  $\nu$ at the origin. 
For the sake of the reader's convenience, we momentarily consider the following standing assumptions on $\nu$, that we improve later on. In what follows, the notation $v^{-1}$ stands for the reciprocal inverse of a bijective function $v$ and should not be confused with the fractional inverse $1/v$. In addition if $\nu$ is radial, we merely identify $\nu(h)=\nu(|h|)$, $h\in \R^d$. 

\begin{enumerate}[\textbf{Assumption} A, wide, labelwidth=!, labelindent=0pt]
\item \hspace{-1.5ex} \textbf{:} \label{item:assump-A} 
The function $\nu$ satisfies the $p$-L\'{e}vy  integrability condition \eqref{eq:plevy-integrable}, is radial and is almost decreasing, i.e., there is $0<\kappa\leq1$ such that 
\begin{align}\label{eq:almost-decreas}\tag{A}
\kappa \nu(|x|) \leq \nu(|y|) \quad\text{for all $|x|\geq |y|$}. 
\end{align}
\item \hspace{-1.5ex} \textbf{:} \label{item:assump-B} 
The mapping $t\mapsto 1/w(1/t)$ is invertible from $[0,\infty)$ to $[0,\infty)$, whose inverse $\phi$ is a Young function (see below for more details) and will be called \emph{the critical Young function associated with $\nu$}. To be more precise, $\phi$ is defined by 
\begin{align}\label{eq:invert-cond}\tag{B}
\phi(t)= \Big( \frac{1}{w(1/t)}\Big)^{-1} \quad \text{equivalently}\quad w(t)= \frac{1}{\phi^{-1}(1/t)} . 
\end{align}
\vspace{1mm}	
\item \hspace{-1.5ex} \textbf{:} \label{item:assump-C} 
The function $\phi_p: [0,\infty) \to [0,\infty)$ with $\phi_p(t)= \phi(t^{1/p})$ is convex, hence a Young function, and satisfies the growth condition: there is $\theta >0$, such that
\begin{align}\label{eq:growth-condition}\tag{C}
\phi_p\big(\theta^p \frac{s}{t}\big)	\leq \frac{\phi_p(s)}{\phi_p(t)} \qquad\text{for all $s\leq t$}.
\end{align}
Clearly, this is equivalent to saying that for all $s\leq t$
\begin{align*}
\phi\big(\theta \frac{s}{t}\big)	\leq \frac{\phi(s)}{\phi(t)} \qquad\text{or equally} \qquad \theta \leq \phi^{-1}\big(\frac{s}{t}\big)\frac{\phi^{-1}(t)}{\phi^{-1}(s)}
\end{align*}
Putting $s= \frac{1}{\phi^{-1}(t')}$ and $t= \frac{1}{\phi^{-1}(s')}$ yields 
\begin{align}\label{eq:growth-condition-bis}\tag{C'}
\phi^{-1}\big(\frac{s'}{t'}\big)	\geq \theta\frac{\phi^{-1}(s') }{\phi^{-1}(t')} \qquad\text{for all $s'\leq t'$}.
\end{align}
\end{enumerate}
\noindent It is worth highlighting that $\phi$ only depends on $\nu, p$ and $d$ but, to alleviate the notations, we keep this implicit. 
Let us now provide basic notions on Young functions and associated Orlicz spaces. 
A thorough and extensive study of Orlicz spaces are carried out in the seminal textbooks \cite{RR91,RR02}. See also the traditional references \cite{Adams,HH19,KR61,KJF77} and the monographs\cite{RGMP16,DHHR11}, where the latter offers a treatise on generalized Orlicz spaces, also known as Musielak-Orlicz spaces, including Lebesgue and Sobolev spaces with variable exponents. Recall that a function $\phi:[0,\infty]\to[0,\infty]$ is convex if, $\phi(s+\tau (t-s)) \leq \phi(s) +\tau(\phi(t) -\phi(s))$ for all $s,t\geq 0$ and $\tau\in [0, 1]$.
\vspace{1mm}
\noindent 	\textbf{Young function:} A convex function $\phi: [0,\infty] \to [0,\infty]$ such that $\phi(0)= 0$ is termed a Young function. 
\noindent Consequently as a Young function $\phi$ is nondecreasing, the mapping $t\mapsto\frac{\phi(t)}{t}$ is nondecreasing on $(0,\infty]$, and, either $\phi\equiv 0$ (i.e. $\phi$ is identically zero) or $\phi(\infty)= \infty$. Moreover, it is well known that $\phi$ is continuous on its effective domain, i.e., on the set of elements in $t\in [0,\infty)$ where $\phi(t)<\infty$. A more advanced calculus, e.g., Jensen's theorem \cite[Theorem 1.3.1]{RR91}, yields the existence of a non-decreasing and right continuous function $b: [0, \infty)\to [0, \infty) $ called the density of $\phi$, such that $\phi(t)= \int_0^t b(s)\d s. $
\noindent This implies that $\phi$ has left and right derivatives that coincide except possibly on a countable set. To avoid unnecessary pathologies, it is customary to also to assume that $\phi$ is neither identically zero nor identically infinite on $(0, \infty)$. 

\vspace{1mm}
\noindent \textbf{Convex conjugate:} To a Young function $\phi$ one associates the convex complementary, also called the convex conjugate, $\widetilde{\phi} :[0,\infty]\to [0,\infty]$, which is simultaneously defined as follows
\begin{align*}
\widetilde{\phi}(t)= \sup\big\{ts-\phi(s):~s>0\big\}=\int_0^t\widetilde{b}(s)\d s.
\end{align*}
Here $\widetilde{b}(t)= \sup\{s>0:~ b(s)<t\}$ is the right inverse of $b$. Clearly, $\widetilde{\phi}$ is also a Young function, i.e., convex and $\widetilde{\phi}(0)=0$. Note that, by virtue of the Fenchel-Moreau theorem, the couple $(\phi, \widetilde{\phi})$ is uniquely defined 
provided that $\phi$ is lower semi-continuous and additionally  we have 
\begin{align*}
\phi(t)= \sup \big\{ts-\widetilde{\phi}(s):~s>0\big\}= \int_0^tb(s)\d s.
\end{align*}
Analogously, the couple $(b, \widetilde{b})$ is uniquely determined and $b$ is also the right inverse of $\widetilde{b}$, i.e., $b(t)= \sup\{s>0:~ \widetilde{b}(s)<t\}.$ Furthermore, if $b$ is strictly increasing then $\widetilde{b}= b^{-1}$, the inverse of $b$. 

\noindent \textbf{$N$-function:} A Young function $\phi: [0,\infty]\to [0,\infty]$ is called a $N-$function (Nice Young function) if its density $b: [0, \infty)\to [0, \infty) $ is nondecreasing, right continuous and satisfies $0<b(t)<\infty $ for $t>0$, $\lim\limits_{t\to0^+ } b(t)=0$ and $\lim\limits_{t\to \infty }b(t)=\infty$. This is equivalent to saying that $\phi$ is continuous, increasing, convex and in addition the mapping $t\mapsto \frac{\phi(t)}{t}$, $t>0$ is increasing and satisfies 
\begin{align}\label{eq:young-asymp}
&\lim_{t\to 0^+}\frac{\phi(t)}{t}= \lim_{t\to \infty}\frac{t}{\phi(t)}= 0.
\end{align} 

\noindent \textbf{Orlicz space:} Next, we write $K^\phi(\R^d)$ and $L^\phi(\R^d)$ respectively to denote the Orlicz class and the Orlicz space with respect to the Young function $\phi$ defined by
\begin{align*}
K^\phi(\R^d)&= \Big\{u: \R^d\to \R\text{ meas.}: ~\int_{\R^d} \phi\big(|u(x)|\big)\d x<\infty \Big\},
\\
L^\phi(\R^d)&= \Big\{u: \R^d\to \R\text{ meas.}:~ \int_{\R^d} \phi\Big(\frac{|u(x)|}{\lambda}\Big)\d x<\infty ~~\text{for some $\lambda>0$}\Big\}.
\end{align*}

\noindent It is worthwhile noting that the Orlicz class $K^\phi(\R^d)$ is a convex set of functions and that $L^\phi(\R^d)$ is the linear hull of $K^\phi(\R^d)$. In addition, $ u\in L^\phi(\R^d) $ if and only if $u\in \lambda K^\phi(\R^d)$ for some $\lambda>0$. The space $L^\phi(\R^d)$ is a Banach space furnished with the Luxemburg norm $\|\cdot\|_{L^\phi(\R^d)}$ defined as the Minkowski functional or gauge of $K^\phi(\R^d)$ by 
\begin{align}\label{eq:luxemburg-norm}
\|u\|_{L^\phi(\R^d)}=\inf \Big\{ \lambda>0~: \int_{\R^d} \phi\Big(\frac{|u(x)|}{\lambda}\Big)\d x\leq 1\Big\}.
\end{align}
Obviously, by Fatou's lemma we have 
\begin{align}\label{eq:modular-lux}
\int_{\R^d} \phi\Big(\frac{|u(x)|}{~~~\|u\|_{L^\phi(\R^d)}}\Big)\d x\leq 1. 
\end{align}

\noindent Beside the Luxemburg norm $\|\cdot\|_{L^\phi(\R^d)}$, we have  the Orlicz norm $|\cdot|_{L^\phi(\R^d)}$, with 
\begin{align*}
\quad |u|_{L^\phi(\R^d)}= \sup\Big\{\int_{\R^d} u(x)v(x)\d x~: \int_{\R^d} \widetilde{\phi} (|v(x)|)\d x\leq 1\Big\}. 
\end{align*}
Moreover, the following comparison holds for all $u\in L^\phi(\R^d)$, 
\begin{align*}
\|u\|_{L^\phi(\R^d)} \leq |u|_{L^\phi(\R^d)} \leq 2\|u\|_{L^\phi(\R^d)}. 
\end{align*} 
\noindent For the Young function $\phi_L(t)= t^p/p$ with 
$1\leq  p<\infty$, the Orlicz space $L^{\phi_L}(\R^d)$ coincides with the well-known Lebesgue space $L^p(\R^d)$. In addition for $p\neq 1,$ we have $\widetilde{\phi}_L(t) = t^{p'}/p'$ with $p'$ satisfying the relation $ pp'= p+p'$. 
\noindent The computation of the Luxemburg norm is not often straightforward. To illustrate this, let $E\subset \R^d$ be a measurable set with finite Lebesgue measure, i.e., $|E|<\infty$ and consider $\mathds{1}_E $ to be its indicator function, i.e., $\mathds{1}_E(x)=1$ if $x\in E$ and $\mathds{1}_E(x)=0$ elsewhere, then 
\begin{align}\label{eq:luxem-indicator}
\|\mathds{1}_E\|_{L^\phi(\R^d)}= \frac{1}{ \phi^{-1}( 1/|E|) }.
\end{align}
Indeed, for $\lambda>0$ we have $\int_{\R^d} \phi(\frac{\mathds{1}_E}{\lambda}) (x) \d x= |E|\phi(\frac1{\lambda})$ which is less than 1 if and only if 
\begin{align*}
\frac{1}{  \phi^{-1}( 1/|E|)}\leq \lambda\quad \text{and hence}\quad \frac{1}{  \phi^{-1}( 1/|E|)}\leq \|\mathds{1}_E\|_{L^\phi(\R^d)}. 
\end{align*}
Therefore, the formula \eqref{eq:luxem-indicator} holds as it suffices to choose $\lambda= \frac{1}{  \phi^{-1}( 1/|E|)}$ to have the reverse inequality. In the same spirit, using Jensen's inequality one establishes that 
\begin{align*}
|\mathds{1}_E|_{L^\phi(\R^d)}=|E| \widetilde{\phi}^{-1} ( 1/|E| ).
\end{align*}
Throughout this note, we only use the Luxemburg norm $\|\cdot \|_{L^\phi(\R^d)}$ defined as in \eqref{eq:luxemburg-norm}. From now, we  assume the Orlicz space $L^\phi(\R^d)$ associated with the critical function $\phi$ is equipped with the norm $\|\cdot\|_{L^\phi(\R^d)}$.  Our main result (see Theorem \ref{thm:main-result2} for a version involving non-radial kernel) reads as follows. 

\begin{theorem}\label{thm:main-result}
Let \ref{item:assump-A}, \ref{item:assump-B} and \ref{item:assump-C} be in force. Define $\Theta_t=  t[2\kappa^2C_p(t)\phi(\frac{\theta}{t})]^{-1/p}$ with $C_p(t)=\frac{t^p-2}{t^p-1}$, $t\geq2$. The following inequality holds
\begin{align}\label{eq:main-equality}
\|u\|_{L^\phi(\R^d)} \leq \Theta_t\Big(\iil_{\R^d\R^d} |u(x)-u(y)|^p\nu(x-y)\d y\d x\Big)^{1/p}\quad\text{for all}\quad u \in L^\phi(\R^d).
\end{align}
Accordingly, the embeddings $W_\nu^p(\R^d) \hookrightarrow L^\phi(\R^d)$
and $W^{1,p}(\R^d) \hookrightarrow L^\phi(\R^d)$ are continuous.
\end{theorem} 

\begin{remark}
$(i)$  Note that $C_p(2)>0$ only if $p>1$. However the case $p=1$ is easily included if we take $t>2$. The constant $\Theta_t$  depends explicitly on $p, \phi, \theta, t, \kappa$ and  implicitly on $d,\nu$ since $\phi$  depends on $d,p,\nu$. Moreover, we point out that the reliance on the free variable $t\geq 2$ is purely cosmetic and solely hinges from the approach used in proving Theorem \ref{thm:main-result}. 

\smallskip 

\noindent $(ii)$ The inequality \eqref{eq:main-equality}  is easily obtained in the particular case $u(x)=\mathds{1}_{E}(x)$, with $|E|<\infty$.  In fact, the relations \eqref{eq:invert-cond} and \eqref{eq:luxem-indicator}  imply  $\|\mathds{1}_E\|_{L^\phi(\R^d)}=w(|E|)$.  Thus by Lemma \ref{lem:inequality-indicator} below we get
\begin{align*}
\|\mathds{1}_E\|_{L^\phi(\R^d)}=w(|E|)\leq 2^{1/p}\kappa^{-2/p}|\mathds{1}_E|_{W^p_\nu(\R^d)},  
\end{align*}
which is an analogue of the inequality \eqref{eq:main-equality} as desired. 
\end{remark}
\noindent The embedding $W^{1,p}(\R^d) \hookrightarrow L^\phi(\R^d)$ reminisces the so-called Trudinger inequality \cite{Mos71,Tru67} which implies that, for a smooth set $\Omega\subset \R^d$ and the Young function $\psi(t)=e^{t^{d/(d-1)}}-1$, $p=d>1,$ the embedding $W^{1,p}(\Omega)\hookrightarrow L^\psi(\Omega)$ is continuous; $L^\psi(\Omega)$ is the Orlicz space on $\Omega$ associated with $\psi$. 
A substantial amount  of results appeared in the wake of\cite{Tru67}, dealing with embedding of Sobolev spaces (eventually of Orlicz-Sobolev Spaces)  into Orlicz spaces. A non-exhaustive list of references on this and related topics includes  \cite{Cia96,Cia04,Cia05,Lul17,PR18,CPS20}. 

\vspace{2mm}
Of course, in accordance to the fractional Sobolev inequality \eqref{eq:frac-sobolev-inequality}, we show later in Example \ref{Ex:fractional} that taking the fractional kernel $\nu(h)= |h|^{-d-sp}$ with $sp<d$, which fits the requirements of Theorem \ref{thm:main-result}, gives $\phi(t)= ct^{p^*_s}$ whose corresponding Orlicz space is $L^{p^*_s}(\R^d).$ See also \cite{ACPS21}  for embeddings of fractional Orlicz-Sobolev spaces into Orlicz type spaces. Let us mention that, in this particular case, the critical exponent $p^*_s$ (or the critical Young function $\phi(t)= ct^{p^*_s}$) can be anticipated using an elementary scaling argument. Unfortunately it is not possible to forecast the critical Young function $\phi$ associated with a general kernel $\nu$, using a scaling argument. For instance, if $0<s_1<s_2<1$, the kernel $\nu(h)= \max(|h|^{-d-s_1p}, |h|^{-d-s_2p})$ or $\nu(h)= \min(|h|^{-d-s_1p}, |h|^{-d-s_2p})$, does not permit the use of a scaling argument. 
\noindent Observe however that we obtain the critical Young function $\phi$ in a more constructive, but still less explicit, manner. The abstract aspect of the kernel $\nu$ under consideration forces a more general setting. Therefore, we will see later in Theorem \ref{thm:main-result2} that, Theorem \ref{thm:main-result} still holds under weakened assumptions on $\nu$. For instance, it appears that the radiality and the  decaying condition \eqref{eq:almost-decreas} can be dispensed. 
\vspace{1mm}

It appears as a natural question to know if it is possible to obtain the analog of the Gagliardo-Nirenberg-Sobolev inequality for functions restricted on an open set $\Omega$ different from $\R^d$. The answer to this important question turns out to be strongly related to the so-called Poincar\'{e}-Sobolev type inequality. Another aim of this note, is to formulate some interplay between Gagliardo-Nirenberg-Sobolev type inequalities, Poincar\'{e}-Sobolev type inequalities and Poincar\'{e} type inequalities; see \cite{HK00} for classical case. We summarize the global idea as follows: if $\Omega\subset \R^d$ is sufficiently smooth then under the assumptions of Theorem \ref{thm:main-result} there is $C>0$ also depending on $\Omega$ such that, 
\begin{align*}
\|u-\mbox{$\fint_\Omega u$ }\|_{L^\phi(\Omega)}\leq C\Big(\iil_{\Omega\Omega} |u(x)- u(y)|^p\nu(x-y)\d y\d x\Big)^{1/p}\quad \text{for all }\quad u\in L^\phi(\Omega). 
\end{align*}

\subsection{Outline}The rest of the paper is structured as follows. In Section \ref{sec:preli}, we comment in details our standing assumptions
by explaining their needfulness, and additionally providing some illustrative examples. Section \ref{sec:main-result} is dedicated to the proof of the main result Theorem \ref{thm:main-result} and its generalization in Theorem \ref{thm:main-result2} with relaxed assumptions. This gives us an opportunity to revisit fractional Gagliardo-Nirenberg-Sobolev inequality with an alternative proof. Finally, in Section \ref{sec:poincare-sobolev} we establish some reciprocity relations between Gagliardo-Nirenberg-Sobolev type inequalities, Poincar\'{e}-Sobolev type inequalities and Poincar\'{e} type inequalities. 

\noindent \textbf{Notations:} Through out, $B(x,r) :=\{y\in \R^d:~|y-x|<r\}$ denotes the open ball with radius $r>0$ and centered at $x\in \R^d$ and its closure is denoted by $\overline{B}(x,r)$. On many estimates, $C>0$ is a generic constant depending on the local inputs. 

\section{Miscellaneous}\label{sec:preli}
In this section we discuss the aforementioned main assumptions and provide at the end, some examples of kernels. We also collect some useful basic results on Orlicz spaces needed in the sequel. 
\subsection{Discussion of the assumptions}
Let us first comment on the aforementioned assumptions and explain their necessity.

\vspace{1mm}

\noindent \textbf{Assumption A:} Although, the class of radial and almost decreasing $p$-L\'{e}vy  integrable kernels is fairly large, we will see later that this assumption can be dropped by the mean of the Schwartz symmetrization rearrangement, see Theorem \ref{thm:inequality-indicator}. 
Nevertheless, having an almost decreasing $p$-L\'{e}vy  integrable kernel, allows us to get a quicker constructive approach of the critical Young function $\phi$ as given in \eqref{eq:invert-cond}.
\noindent Beside this, the $p$-L\'{e}vy  integrability condition, i.e., $\nu\in L^1(\R^d, 1\land|h|^p)$, can neither be improved nor completely dropped. Indeed, this condition renders the space $W_\nu^p(\R^d)$ none trivial, i.e., $W_\nu^p(\R^d)\neq\{0\}$ and  more consistent, in a sense that it warrants the space $W_\nu^p(\R^d)$ to contain at least smooth functions of compact support. In fact, as mentioned earlier it can be shown that the $p$-L\'{e}vy  integrability condition is very sharp and self-generated in the  sense that $C_c^\infty(\R^d) \subset W_\nu^p(\R^d)$ if and only if $\nu\in L^1(\R^d, 1\land|h|^p)$. This is the content of the following result whose proof can be found in \cite[Section 4]{Fog23s}; see also the variant in\cite{FK22}. 

\begin{theorem}\label{thm:charac-plevy-radial-bis} 
The following assertions are equivalent. 
\begin{enumerate}[$(i)$]
\item The $p$-L\'{e}vy condition \eqref{eq:plevy-integrable} holds, i.e. $\nu\in L^1(\R^d,1\land|h|^p)$.
\item The embedding $W^{1,p}(\R^d)\hookrightarrow W^p_\nu(\R^d)$ is continuous. 
\item $|u|_{W^p_\nu(\R^d)}<\infty$ for all $u\in W^{1,p}(\R^d)$. 
\item $|u|_{W^p_\nu(\R^d)}<\infty$ for all $u\in C_c^\infty(\R^d)$. 
\item The space $W^p_\nu(\R^d)$ is nontrivial, i.e.,  $W^p_\nu(\R^d)\neq \{0\}$. 
\end{enumerate}
Moreover, this also remains true when $p=1$ with $BV(\R^d)$ in place of $W^{1,1}(\R^d)$.
\end{theorem}

\noindent As illustrated by the next proposition, see details in  \cite[Proposition 2.14]{Fog23} or \cite{Fog20}, the $p$-L\'{e}vy  integrability draws a borderline for which a space  $W_\nu^p(\R^d)$ is trivial or not. 

\begin{proposition} Let $\nu: \R^d\to [0,\infty]$ be symmetric. The following assertions hold true. 
\begin{enumerate}[$(i)$]
\item If $\nu\in L^1(\R^d)$ then $W_\nu^p(\R^d)= L^p(\R^d)$ with equivalence in norm. 
\item Define $C_\delta=\int_{B(0,\delta)} |h|^p\nu(h)\d h$. If $\nu$ is radial and $C_\delta=\infty$ for all $\delta>0$, thus $\nu \not\in L^1(\R^d, 1\land |h|^p)$, then all smooth functions contained in $W_\nu^p(\R^d)$ are constants. 
\item 
If $\nu$ is radial then for $u\in W^{1,p} (\R^d)$ there is $\delta=\delta(u)>0$ depending on $u$ such that 
\begin{align}\label{eq:equiv-sobolev}
2^{-p}K_{d,p} C_\delta|u|_{W^{1,p} (\R^d)} \leq |u|_{W^{p}_\nu (\R^d)} \leq 2\|\nu\|^{1/p}_{L^1(\R^d, 1\land |h|^p)}\|u\|_{W^{1,p} (\R^d)}.
\end{align}
\end{enumerate}
\end{proposition}
\noindent \textbf{Warning!} The estimates in  \eqref{eq:equiv-sobolev} do not imply that $W^{1,p}(\R^d) = W^p_\nu(\R^d)$. 
\vspace{2mm}

\noindent \textbf{Assumption B:} 
The convexity part of the \ref{item:assump-B} turns out to be weaker than that of \ref{item:assump-C}. Indeed, the function $t\mapsto \phi_p(t) = \phi(t^{1/p})$ being a Young function and hence convex and nondecreasing implies that $\phi(t) = \phi_p(t^p)$ is also convex and hence a Young function. In  fact the condition that $\phi_p$ is convex, can be viewed as a fair analog of the condition that $\frac{1}{p^s_*}>0$ in the standard situation where $\nu(h)= s(1-s)|h|^{-d-sp}$.  Furthermore, it is natural to require the function $\phi$ to be invertible as it rules out pathological functions. Note that, assuming $\phi$ is a Young function, one views from \eqref{eq:invert-cond} that $\phi(0)= 0=w(0)$ and $\phi(\infty)= \infty=w(\infty)$ and hence that $t\mapsto \phi(t)$ is invertible from $ [0,\infty)$ to $[0,\infty)$ if and only if $r\mapsto w^p(r)$ is. 
\noindent Therefore, \ref{item:assump-B} is somewhat a great accessory to define the critical Young function $\phi$ and the associated Orlicz space $L^\phi(\R^d)$. 

\vspace{2mm}

\noindent \textbf{Assumption C:} The \ref{item:assump-C} essentially constitutes the most fundamental and quite vital property needed on $\phi$ in order to establish our main result. Next, we explain how the \ref{item:assump-C} globally infers certain growth conditions on $\phi$. First of all, the growth condition \eqref{eq:growth-condition} is clearly equivalent to saying that
\begin{align}\label{eq:growth-phip-bis}
\phi\big(\theta \frac{s}{t}\big)	\leq \frac{\phi(s)}{\phi(t)} \quad \text{for all $s\leq t$}\quad\text{or equally} \quad \theta \leq \phi^{-1}\big(\frac{s}{t}\big)\frac{\phi^{-1}(t)}{\phi^{-1}(s)} \quad \text{for all $s\leq t$}.
\end{align}
The latter suggests that the growth behavior of $\phi$ is not far from that of a polynomial growth \cite{Mag85}; see for instance Proposition \ref{prop:growth-phi} below. 
This behavior can be expected, considering  the inequality of interest \eqref{eq:main-equality} in Theorem \ref{thm:main-result}. Since at the first glance, in comparison with the fractional Sobolev space $W^{s,p}(\R^d)$, one can expect that the space $W^{p}_\nu (\R^d)$ is embedded in another Lebesgue space. Nevertheless, it is worth noticing that fractional kernels of the form $\nu(h)=|h|^{-d-sp}$ with $s\in (0, 1)$ are the only radial functions satisfying \ref{item:assump-A}, \ref{item:assump-B} and \ref{item:assump-C} yielding  polynomial critical Young functions of the form $\phi(t)= ct^q$; see Example \ref{Ex:fractional} and Theorem \ref{thm:fractional} below.  
Observe that letting $s=t\tau$ with $0\leq \tau\leq 1$ then the condition \eqref{eq:growth-condition} (see also \eqref{eq:growth-phip-bis}) is also to equivalent to saying that $\phi$ satisfies sub-multiplicative condition 
\begin{align*}
\phi(\theta \tau)\phi(t)\leq \phi(t\tau )\qquad\text{for all $t\geq 0$ and $\tau\in [0,1]$}
\end{align*}
or that $\phi^{-1}$ satisfies the sub-multiplicative condition 
\begin{align*}
\theta\phi^{-1}(\tau t)\leq\phi^{-1}(\tau ) \phi^{-1}(t )\qquad\text{for all $t\geq 0$ and $\tau\in [0,1]$}.
\end{align*}
\noindent Let us now highlight some facts about the Young function $\phi_p(t)= \phi(t^{1/p})$. Assume $\phi$ is given as in \eqref{eq:invert-cond} then taking $r= \phi(t)$ for $t>0$, in virtue of \eqref{eq:defw1} and \eqref{eq:invert-cond} we obtain

\begin{align}\label{eq:rate-function}
\frac{\phi(t)}{t^p}&= \frac{r}{[\phi^{-1}(r)]^p}=rw^p(1/r) 
= \int_{|h|^d\geq\frac{1}{c_d\phi(t)}}\nu(h) \d h. 
\end{align}
\noindent In particular, if $\lim\limits_{t\to 0^+}\phi(t)=0$ and $\lim\limits_{t\to \infty }\phi(t)=\infty$ then 
\begin{align}\label{eq:phip-asymp}
\lim_{t\to \infty}	\frac{\phi(t)}{t^p}= 	\int_{\R^d}\nu(h) \d h \quad\text{and}\quad \lim_{t\to 0^+}	\frac{\phi(t)}{t^p}= \lim_{r\to \infty}\int_{|h|\geq r}\nu(h) \d h=0.
\end{align}

\noindent The convexity of $\phi_p(t)= \phi(t^{1/p})$ induces  certain geometry growths near $0$ and near  $\infty$. 
\begin{proposition}\label{prop:growth-phi} Assume $t\mapsto \phi_p(t)= \phi(t^{1/p})$ is an invertible Young function. 
\begin{enumerate}[$(i)$]
\item $\phi$ is also a Young function. 
\item The mappings $t\mapsto \phi(t)$ and $t\mapsto \frac{\phi(t)}{t^p}$ are increasing. 
\item If $1<p<\infty$ then $\phi$ is an $N$-function. 
\item Let $\delta_0=\frac{\phi(t_0)}{t_0^p}$ for fixed $t_0>0$ then we have 
\begin{align*}
	\phi(t)\leq \delta_0 t^p\quad\text{if ~~ $~~0\leq t\leq t_0$} \qquad \text{and}\qquad 
	\phi(t)\geq \delta_0 t^p\quad\text{if ~~$ ~~t\geq t_0$}. 
\end{align*}
\item Let $\delta'_0=\frac{w^p(r_0)}{r_0}$ for fixed $r_0>0$ then we have 
	\begin{align*}
\int_{B^c(0, \eta(r))} \hspace{-4ex}\nu(h) \d h=\frac{w^p(r)}{r} \geq \delta'_0\quad\text{if~~ $~~ 0\leq r\leq r_0$}\quad\text{and}\quad \int_{B^c(0, \eta(r))}\hspace{-4ex} \nu(h) \d h=\frac{w^p(r)}{r} \leq \delta'_0 \quad\text{ if~~ $~~r\geq r_0$}.
	\end{align*}
\item 	If $\nu\in L^1(\R^d, 1\land |h|^p)$ and $\nu\not \in L^1(\R^d)$ then $\phi_p$ is an $N$-function, equally we have 
	\begin{align*}
	\lim_{t\to 0^+}	\frac{\phi(t)}{t^p}= \lim_{t\to \infty}	\frac{t^p}{\phi(t)}=0.
	\end{align*}
\item 	If $\nu\in L^1(\R^d, 1\land |h|^p)$ is radial, then $r\mapsto \frac{d}{ dr} \Big(\frac{w^p(r)}{r}\Big)$ satisfies the differential equation 
\begin{align}\label{eq:ode-w}
\nu(\eta(r))= -\frac{d}{ dr} \Big(\frac{w^p(r)}{r}\Big)\quad\text{and}\quad \lim_{r\to \infty}	\frac{w^p(r)}{r}= 0.
\end{align}
\end{enumerate}
\end{proposition}

\begin{proof}
As an invertible Young function $t\mapsto\phi_p(t) $ is increasing and hence the assertion $(i)$ is clear since $\phi(t) =\phi_p(t^p)$ and $t\mapsto t^p$ is also convex. In view of proving $(ii)$, observe that $\phi$ is increasing as it is an invertible Young function. Since $\phi_p$ is convex and invertible with $\phi_p(0)=0$, we get 
\begin{align*}
\phi_p(s^p)= \phi_p(\frac{s^p}{t^p} t^p)<\frac{s^p}{t^p} \phi_p(t^p) \quad\text{that is}\quad 
\frac{\phi(s)}{s^p} <\frac{\phi(t)}{t^p} ,\quad \text{for all $s<t$}. 
\end{align*}
The assertion $(iii)$ follows from $(i)$ and $(ii)$. The assertions $(iv)$ and $(v)$ are consequences of $(ii)$. Whereas, $(vi)$ clearly follows from \eqref{eq:phip-asymp}. Differentiating the relation \eqref{eq:defw1} gives $(vii)$ since 
\begin{align*}
\frac{w^p(r)}{r} = \int_{B^c(0, \eta(r))} \nu(h)\d h
= dc_d\int_{\eta(r)}^\infty \nu(\tau)\tau^{d-1}\d \tau = \int_{r}^\infty \nu(\eta(\tau'))\d \tau'. 
\end{align*}
\end{proof} 

\subsection{Basics on Orlicz spaces}Here we collect some basics on Orlicz spaces. We borrow the following result from{\cite[Theorem 5.1.3]{RR91}}. 

\begin{theorem}\label{thm:orlicz-emb}
Let $D\subset\R^d$ be measurable and let $\phi_i,~i=1,2$ be a pair of Young functions. If $|D|<\infty$ $($resp. $|D|=\infty)$ and $\phi_1(t)\leq \phi_2(ct)$ for all $t\geq t_0$ for some $c>0$ and $ t_0>0$ $($resp. $t_0=0)$ then the embedding $L^{\phi_2}(D) \hookrightarrow L^{\phi_1}(D)$ is continuous. The converse holds true as well. 
\end{theorem}

\begin{proof}The case $|D|=\infty$ and $t_0=0$ is straightforward and one has $\|u\|_{L^{\phi_1}(D)}\leq c\|u\|_{L^{\phi_2}(D)}.$ Now, assume $|D|<\infty$, for $u\in L^{\phi_2}(D)$, consider $A=\{x\in D: |u(x)|\leq ct_0 \|u\|_{L^{\phi_2}(D)}\}$ and put $T= \phi_1(t_0)|D|+1$. Since $\phi_1(\frac{t}{T})\leq \frac{1}{T}\phi_1(t)$ for $t>0$, recalling \eqref{eq:modular-lux}, one gets 
\begin{align*}
\int_D \phi_1\Big(\frac{|u(x)| }{Tc \|u\|_{L^{\phi_2}(D)}}\Big)\d x
&= \int_A \phi_1\Big(\frac{|u(x)|}{Tc \|u\|_{L^{\phi_2}(D)}}\Big)\d x+ \int_{D\setminus A} \phi_1\Big(\frac{|u(x)|}{Tc \|u\|_{L^{\phi_2}(D)}}\Big)\d x\\
&\leq \frac{1}{T}\Big(\phi_1(t_0)|A|+ \int_{D\setminus A} \phi_2\Big(\frac{|u(x)|}{ \|u\|_{L^{\phi_2}(D)}}\Big)\d x\Big)\\
&\leq \frac{1}{T}\big(\phi_1(t_0)|D|+1\big)=1. 
\end{align*}
Accordingly, this implies that $\|u\|_{L^{\phi_1}(D)}\leq cT\|u\|_{L^{\phi_2}(D)}.$
\noindent Conversely, assume there is no constant $c>0$ such that $\phi_1(t) \leq \phi_2(ct)$ for all $t>t_0>0$. Then one can construct an increasing sequence $0<t_0<\cdots <t_k<t_{k+1}\cdots $ such that $\phi_1(t_k)>\phi_2(2^kk^2t_k)$. In particular, $\phi_2(t_k)>0$ and the convexity implies $\phi_1(t_k)>2^k \phi_2(k^2t_k)$. Fix $D_0\subset D$ such that $0<|D_0|<\infty$ and let $D_k\subset D_0$ be disjoint measurable sets such that $|D_k|>0$ and 
\begin{align*}
|D_k|=\frac{\phi_2(t_1) |D_0|}{2^k\phi_2(k^2t_k)},\quad\text{and hence } \quad
\sum_{k=1}^\infty 	|D_k|<|D_0|.
\end{align*}

\noindent We claim that the function $u = \sum_{k=1}^\infty k t_k\mathds{1}_{D_k}$ (which is supported in $D_0$) belongs in $L^{\phi_2}(\R^d)$ but not in $L^{\phi_1}(\R^d)$. Indeed, for any integer $n\geq 1$ we have 
\begin{align*}
\int_D \phi_2(n u(x))\d x &= \sum_{k=1}^\infty \phi_2(n k t_k)|D_k|\leq \sum_{k=1}^n \phi_2(n k t_k)|D_k| + \sum_{k\geq n+1} \phi_2( k^2 t_k)|D_k|\\
&= \sum_{k=1}^n \phi_2(n k t_k)|D_k| + \phi_2(t_1) |D_0|\sum_{k\geq n+1}\frac{1}{2^k}<\infty. 
\end{align*}
Thus  $u\in L^{\phi_2}(\R^d)$. However, for any $\eps>0$, recalling that $\phi_1(t_k)>2^k \phi_2(k^2t_k)$ we have 
\begin{alignat*}{2}
\int_D \phi_1(\eps u(x))\d x &\geq \sum_{k\geq \frac{1}{\eps} } \phi_1(\eps k t_k)|D_k|&&\geq\sum_{k\geq \frac{1}{\eps} } \phi_1( t_k)|D_k| \quad \text{(since $~~\eps k\geq 1$)}\\
& \geq \sum_{k\geq \frac{1}{\eps} } 2^k \phi_2(k^2t_k)|D_k| &&= \phi_2(t_1) |D_0| \sum_{k\geq \frac{1}{\eps} } 1= \infty. 
\end{alignat*}
This implies $u\not\in L^{\phi_1}(\R^d)$. Hence $ L^{\phi_2}(\R^d)\not\hookrightarrow L^{\phi_1}(\R^d)$.  The proof is complete.
\end{proof}

\smallskip

\begin{corollary}\label{cor:embed-loc} Assume $t\mapsto \phi_p(t)= \phi(t^{1/p})$ is an invertible Young function, then the embedding 
$L^\phi(\R^d)\hookrightarrow L^p_{\operatorname{loc}}(\R^d)$ is continuous. 

\end{corollary}

\begin{proof}
By Proposition \ref{prop:growth-phi}, $\phi(t)\geq\delta_0 t^p $ for all $t\geq 1$. The claim follows  from Theorem \ref{thm:orlicz-emb}. 
\end{proof}

\noindent
Let us recall without proof the characterization of  the intersection and the sum of Orlicz spaces, see \cite{RGMP16}.
Let $\phi_i,~i=1,2$ be a pair of Young functions. 
The function $\phi_{\max}(t)= \max(\phi_1(t), \phi_2(t)) $ is a Young function and  $L^{\phi_1}(\R^d)\cap L^{\phi_2}(\R^d) =L^{\phi_{\max}} (\R^d)$. Moreover, if we equip $L^{\phi_1}(\R^d)\cap L^{\phi_2}(\R^d) $ with $\|\cdot\|_{L^{\phi_1}(\R^d)\cap L^{\phi_2}(\R^d)}= \max(\|\cdot\|_{L^{\phi_1}(\R^d)},\|\cdot\|_{ L^{\phi_2}(\R^d) })$, there holds that
\begin{align}\label{eq:orlicz-inter-max}
\frac12\|u\|_{ L^{\phi_{\max}} (\R^d) }
\leq \|u\|_{L^{\phi_1} (\R^d)\cap L^{\phi_2} (\R^d)}\leq \|u\|_{L^{\phi_{\max}}(\R^d)}. 
\end{align} 		
\noindent In addition, for any other Young function $\psi$ such that $\psi(t)\leq \phi_{\max}(ct)$ for some $c>0$, then by Theorem \ref{thm:orlicz-emb} the embedding $L^{\phi_1}(\R^d)\cap L^{\phi_2}(\R^d) \hookrightarrow L^\psi(\R^d)$ is continuous.
Analogously, for the function $\phi_{\min}(t)= \min(\phi_1(t), \phi_2(t)) $ we have $L^{\phi_1}(\R^d)+L^{\phi_2}(\R^d)=L^\phi(\R^d)$. Note however, that $\phi_{\min}$ is not necessarily convex and thus, is identified with its greatest convex minorant $\phi_{\min_0}$,
\begin{align*}
\phi_{\min_0}(t)=\int_0^t\frac{\phi_{\min}(s)}{s}\d s= \int_0^t\frac{\min(\phi_1(s), \phi_2(s))}{s}\d s. 
\end{align*}
So that, $\phi_{\min_0}(t)\leq \phi_{\min}(t)\leq\phi_{\min_0}(2t)$ and hence $L^{\phi_{\min_0}}(\R^d)=L^{\phi_{\min}}(\R^d)$. In what follows, we identify $\phi_{\min}$ with $\phi_{\min_0}.$
Moreover,  one can check that 
\begin{align*}
& \frac14\|u\|_{L^{\phi_{\min}}(\R^d)}\leq \|u\|_{L^{\phi_1}(\R^d)+L^{\phi_2}(\R^d) }\leq 2\|u\|_{L^{\phi_{\min}}(\R^d)},
\end{align*} 
where we recall that  $\|\cdot\|_{L^{\phi_1}(\R^d)+L^{\phi_2}(\R^d) }$ is the natural norm on $L^{\phi_1}(\R^d)+L^{\phi_2}(\R^d)$
\begin{align*}
&\|u\|_{L^{\phi_1}(\R^d)+L^{\phi_2}(\R^d) }
=\inf\big\{\|u_1\|_{L^{\phi_1}(\R^d)}+ \|u_2\|_{L^{\phi_2}(\R^d) }: \,\, u= u_1+u_2, \, u_i\in L^{\phi_1}(\R^d)\big\}.
\end{align*}
\noindent In the particular case of Lebesgue spaces,  $L^{p_1}(\R^d)\cap L^{p_2}(\R^d)$ resp. $L^{p_1}(\R^d)+ L^{p_2}(\R^d) $, $1\leq p_1\leq p_2$,  is the Orlicz space associated with the function $t\mapsto \max(t^{p_1}, t^{p_2})$ resp. $t\mapsto \min(t^{p_1}, t^{p_2})$. Moreover, observe that we have $q\in [p_1, p_2]$ if and only if  $ \min(t^{p_1}, t^{p_2})\leq t^q\leq \max(t^{p_1}, t^{p_2})$ for all $t>0$. Whence we deduce from Theorem \ref{thm:orlicz-emb} that  $L^{p_1}(\R^d)\cap L^{p_2}(\R^d) \hookrightarrow L^q(\R^d)\hookrightarrow L^{p_1}(\R^d)+L^{p_2}(\R^d) $.

\vspace{2mm}

\subsection{Examples of kernels} Let us provide examples for which the main inequality \eqref{eq:main-equality} holds. First of all, we deal with the standard fractional kernel $\nu(h)= |h|^{-d-sp}$, which allows us to recover the classical fraction Gagliardo-Nirenberg-Sobolev inequality. 

\begin{example} \label{Ex:fractional}
For $ s\in (0,1)$, consider the kernel $\nu(h)= |h|^{-d-sp}$, $h\neq 0$ so that $W_\nu^p(\R^d)= W^{s,p}(\R^d)$. A painless computation through polar coordinates in \eqref{eq:defw1} and \eqref{eq:invert-cond} yields that 
\begin{align}\label{eq:fractional-case}
w(r) = \gamma^{1/p}_s\,\, r^{1/p^*_s}\qquad \text{and}\qquad \phi(t) =\gamma^{p^*_s/p}_s \,\,t^{p^*_s},
\end{align}
\noindent where, recalling $c_d=|B(0,1)|$, we set
\begin{align*}
\frac{1}{p^*_s}=\frac{1}{p}-\frac{s}{d} \qquad \text{and}\qquad\gamma_s= \frac{d c_d^{1+\frac{sp}{d}} }{sp} . 
\end{align*}
Observe that $1/p^*_s>0$ if and if $p^*_s\geq p\geq 1$ and, hence if and only if $\phi_p(t)= \phi(t^{1/p}) =\gamma_s^{p^*_s/p} t^{p^*_s/p}$ is convex. Moreover, for all $s,t>0$ we have 
\begin{align*}
\phi_p\big(\theta^p\frac{s}{t}\big)= \frac{\phi_p(s)}{\phi_p(t)}\quad\text{with }\quad \theta =\frac{1}{ \gamma_s^{1/p}}. 
\end{align*}
Whence, \ref{item:assump-A}, \ref{item:assump-B} and \ref{item:assump-C} are fulfilled provided that $1/p^*_s>0.$ 
\end{example}

\begin{example} Assume $\nu\in L^1(\R^d)$ is radial and has full support so that $\phi$ exists. The relation \eqref{eq:rate-function} implies $\phi(t)\leq \|\nu\|_{L^1(\R^d)} t^p $ for all $t>0$. Whence, according to Theorem \ref{thm:orlicz-emb} we get the continuous embedding $L^p(\R^d)\hookrightarrow L^\phi(\R^d)$ and we have
\begin{align*}
\|u\|_{L^\phi(\R^d)} \leq \|\nu\|^{1/p}_{L^1(\R^d)} \|u\|_{L^p(\R^d)}\quad \text{for all $u\in L^\phi(\R^d)$.}
\end{align*}
Together with Corollary \ref{cor:embed-loc}, we get the continuous embeddings $L^p(\R^d)\hookrightarrow L^\phi(\R^d)\hookrightarrow L^p_{\operatorname{loc}}(\R^d)$. 
\end{example}

\begin{example}
\noindent For concrete examples, consider the Young function $\phi^a(t) = \ln (a+ e^{t^p})-\ln(a+1)$ here $a>0$ is a fixed parameter. Note that $\phi(t)\leq t^p $, for all $t>0$. Clearly $\phi^a_p(t) = \phi^a(t^{1/p}) =\ln (a+ e^{t})-\ln(a+1)$ is also a Young function. Moreover, each $\phi^a$ satisfies \eqref{eq:growth-condition} with $\theta=1$. Last one defines the kernel $\nu^a\in L^1(\R^d) $ associated with $\phi^a$ through the relation 
\begin{align*}
\nu^a(\eta(r)) = -\frac{d}{d r} \Big(\frac{1}{r\xi^a(r)}\Big) \quad\text{with $\xi^a(r) := [\phi^a_p]^{-1}(1/r)= \ln((a+1)e^{1/r}- a)$}. 
\end{align*}
\noindent Each $\nu^a$ fulfills \ref{item:assump-A}, \ref{item:assump-B} and \ref{item:assump-C}.
\end{example}

\noindent The next example exhibits a situation where the growth condition \eqref{eq:growth-condition} fails but the inequality \eqref{eq:main-equality} still holds true with a possibly different constant. 
\begin{example}\label{Ex:fractional-max}
Fix $0<s_1<s_2<1$, following the notations of Example \ref{Ex:fractional}, we define the Young function $\phi(t)= \max(t^{p^*_{s_1}} , t^{p^*_{s_2}})$, with $1/p^*_{s_i} >0$, $i=1,2$ so that $L^\phi(\R^d) = L^{p^*_{s_1}}(\R^d)\cap L^{p^*_{s_2}}(\R^d)$. Clearly, $\phi_p(t) = \phi(t^{1/p})$ is convex since $p^*_{s_2}>p^*_{s_1}> p$.  Moreover the relation \eqref{eq:invert-cond} gives 
\begin{align*}
w^p(r)= \frac{1}{\phi_p^{-1}(1/r)}= \max(r^{p/p^*_{s_1}} , r^{p/p^*_{s_2}}). 
\end{align*}

\noindent Now, we differentiate the relation \eqref{eq:defw1} and put $\rho=\eta(r),$ equally $r=c_d\rho^d$ to obtain that
\begin{align*}
\nu(\rho)= \nu(\eta(r))= -\frac{d}{ dr} \Big(\frac{w^p(r)}{r}\Big)
= \begin{cases} \frac{1}{	\gamma_{s_2}}\rho^{-d-s_2p}& \text{if}\,\, \rho<\eta(1),\\
\frac{1}{	\gamma_{s_1}}\rho^{-d-s_1p} & \text{if}\,\, \rho\geq \eta(1).
\end{cases}
\end{align*}

\noindent Whence the kernel $\nu$ is given by 
\begin{align*}
\nu(h)= \frac{1}{	\gamma_{s_2}}\mathds{1}_{B(0, \,\eta(1) ) }(h) \,|h|^{-d-s_2p}+ \frac{1}{	\gamma_{s_1}}\mathds{1}_{B^c(0, \,\eta(1) ) }(h) \, |h|^{-d-s_1p}. 
\end{align*} 
\noindent One easily finds that $c_1\nu(h)\leq \max( |h|^{-d-s_1p}, |h|^{-d-s_2p})\leq |h|^{-d-s_1p}+ |h|^{-d-s_2p}\leq c_2\nu(h)$ for some constants $c_1, c_2>0$.
Note however, that the growth condition \eqref{eq:growth-condition} cannot hold here, i.e., there is no constant $\theta>0$ such that 
\begin{align*}
\phi\big(\theta\frac{s}{t}\big)	\leq \frac{\phi(s)}{\phi(t)} \quad\text{for all $s\leq t$}.
\end{align*}
It suffices to take $s=1$ and tend $t\to\infty$ to observe a contradiction. Nevertheless, according to Example \ref{Ex:fractional}, Theorem \ref{thm:main-result} applies on the kernels $|h|^{-d-s_ip}\leq c_2 \nu(h)$, $i=1,2$ and since by \eqref{eq:orlicz-inter-max}, $L^\phi(\R^d) = L^{p^*_{s_1}}(\R^d)\cap L^{p^*_{s_2}}(\R^d)$ and $\|u\|_{L^\phi(\R^d)} \leq 2\max\big(\|u\|_{L^{p^*_{s_1}}(\R^d)}, \|u\|_{L^{p^*_{s_2}}(\R^d)}\big)$, we get
\begin{align*}
\begin{split}
\|u\|_{L^\phi(\R^d)} 
\leq C\Big(\iil_{\R^d\R^d} |u(x)-u(y)|^p\nu(x-y)\d y\d x\Big)^{1/p}
\end{split}\quad\text{for all}\quad u \in L^\phi(\R^d).
\end{align*}
In other words, inequality \eqref{eq:main-equality} still holds despite the failure of the growth condition \eqref{eq:growth-condition}. 
\end{example}

\noindent The next example shows that the lack of convexity can sometimes be rectified. 
\begin{example} \label{Ex:fractionalmin}
Fix $0<s_1<s_2<1$, considering the notations of Example \ref{Ex:fractional} define $ \phi(t)= \min(t^{p^*_{s_1}} , t^{p^*_{s_2}})$, with $p^*_{s_i}>0,\, i=1,2$. Note that $\phi_p(t) = \phi(t^{1/p})$ is not necessarily convex. However one rectifies this deficiency by defining 
\begin{align*}
\phi_{\min}(t)= \int_0^t\frac{\min(s^{p^*_{s_1}}, s^{p^*_{s_2}})}{s}\d s \quad\text{so that }\quad \phi_{\min}(t^{1/p})=\frac1p \int_0^t\frac{\min(s^{p^*_{s_1}/p}, s^{p^*_{s_2}/p})}{s}\d s. 
\end{align*}
\noindent Since, $p^*_{s_2}\geq p^*_{s_1}\geq p$, one readily obtains that $t\mapsto \phi_{\min}(t^{1/p})$ is convex. Furthermore, 
\begin{align*}
& \phi\big(\frac{s}{t}\big)	\leq \frac{\phi(s)}{\phi(t)} \quad\text{for all $s\leq t$ \quad and }\quad \phi_{\min}(t)\leq \phi(t)\leq\phi_{\min}(2t)\quad\text{for all $t\geq 0$}. 
\end{align*}
\noindent Combining altogether implies that 
\begin{align*}
\phi_{\min}\big(\frac{s}{2t}\big)\leq \phi\big(\frac{s}{2t}\big)	\leq \frac{\phi(s/2)}{\phi(t)}\leq \frac{\phi_{\min}(s)}{\phi_{\min}(t)} \quad\text{for $s\leq t$.} 
\end{align*}
\noindent It turns out that, $\phi_{\min}$ satisfies \eqref{eq:growth-condition} with $\theta=\frac{1}{2}$ and thus the \ref{item:assump-C}. Therefore, it is fair to identify $\phi$ with $\phi_{\min}$ so that $L^\phi(\R^d)= L^{\phi_{\min}} (\R^d)= L^{p^*_{s_1}}(\R^d)+L^{p^*_{s_2}}(\R^d)$. Next, we find the kernel associated with $\phi(t)= \min(t^{p^*_{s_1}} , t^{p^*_{s_2}})$. The relation \eqref{eq:invert-cond} yields
\begin{align*}
w^p(r)= \frac{1}{\phi_p^{-1}(1/r)}= \min(r^{p/p^*_{s_1}} , r^{p/p^*_{s_2}}). 
\end{align*}

\noindent Using the relation \eqref{eq:ode-w} and put $\rho=\eta(r),$ equally $r=c_d\rho^d$ to obtain
\begin{align*}
\nu(\rho)= \nu(\eta(r))= -\frac{d}{ dr} \Big(\frac{w^p(r)}{r}\Big)
&= \begin{cases} \frac{1}{	\gamma_{s_2}} \rho^{-d-s_2p}& \text{if}\,\, \rho\geq\eta(1),\\
\frac{1}{	\gamma_{s_1}}\rho^{-d-s_1p} & \text{if}\,\, \rho< \eta(1).
\end{cases}
\end{align*}
\noindent Whence, the kernel $\nu$ is given by 
\begin{align*}
\nu(h)= \frac{1}{	\gamma_{s_1}}\mathds{1}_{B(0, \,\eta(1) ) }(h) \,|h|^{-d-s_1p}+ \frac{1}{	\gamma_{s_2}}\mathds{1}_{B^c(0, \,\eta(1) ) }(h) \, |h|^{-d-s_2p}. 
\end{align*}
\noindent One easily finds that $c_1\nu(h)\leq \min( |h|^{-d-s_1p}, |h|^{-d-s_2p}) \leq c_2\nu(h)$ for some constants $c_1, c_2>0$. Thus, identifying $\phi$ and $\phi_{\min}$ \ref{item:assump-A} , \ref{item:assump-B} and \ref{item:assump-C} are satisfied. 
\end{example}
\begin{remark}
There are two key geometric observations emanating from Example \ref{Ex:fractional}, Example \ref{Ex:fractional-max} and Example \ref{Ex:fractionalmin}. Firstly, modifying the $p$-L\'{e}vy  integrable kernel $\nu$ at the origin or at the infinity may not change the topology of the nonlocal Sobolev space $W_\nu^p(\R^d)$. 

\vspace{1mm}
\noindent Secondly, the geometric behavior of the kernel $\nu$ at the origin or at the infinity truly governs that of the associated critical function $\phi$ and hence influences the topology of the Orlicz space $L^\phi(\R^d)$. In other words
a perturbation of the kernel $\nu$ at the origin or at the infinity can drastically change the resulting associated Orlicz space (or associated critical function). This geometric behavior also reads  through the relation \eqref{eq:rate-function} which implies that
\begin{align*}
\lim_{t\to \infty}	\frac{\phi(t)}{t^p}= 	\int_{|h|\geq \eta(\frac{1}{\phi(\infty)}) }\nu(h) \d h \quad\text{and}\quad \lim_{t\to 0^+}	\frac{\phi(t)}{t^p}=\int_{|h|\geq \eta(\frac{1}{\phi(0)})}\nu(h) \d h.
\end{align*}
\end{remark}

\section{Main results}\label{sec:main-result} 
\noindent 
With a view to establish our main result, 
we need auxiliary results that are the milestones to prove Theorem \ref{thm:main-result}. We begin with the following important lemma. 

\begin{lemma}\label{lem:inequality-indicator}
Assume $\nu$ is almost decreasing, i.e., satisfies \eqref{eq:almost-decreas}. The following estimate holds
\begin{align*}
\essinf_{x\in\R^d}	\int_{E^c} \nu(x-y)\d y \geq 
\kappa^2 \frac{w^p(|E|)}{|E|}, \qquad \forall\,\, \text{$E\subset \R^d$ meas., $|E|<\infty$}.
\end{align*}

\end{lemma}

\begin{proof}
Let the ball $B(0,r_E)$ centered at the origin with radius $r_E$, be the symmetric rearrangement of $E$ that is $|E|=|B(0,r_E)|$, equally $r_E= (\frac{|E|}{c_d})^{1/d}= \eta(|E|)$ where $c_d=|B(0,1)|$. Noticing that $A\setminus B=A\setminus (A\cap B)$, one gets
\begin{align*}
\big|B(x,r_E)\setminus E\big|= \big|B(x,r_E)\big|- \big| E\cap B(x,r_E) \big|= \big|E\big|- \big| E\cap B(x,r_E) \big|= \big|E\setminus B(x,r_E) \big|. 
\end{align*}

\noindent Accordingly, using the fact that $\nu $ is almost decreasing, we get the sought estimate as follow
\begin{align*}
\int_{E^c} \nu(x-y)\d y&= 	\int_{E^c\cap B(x, r_E)} \nu(x-y)\d y+	\int_{E^c\cap B^c(x, r_E)} \nu(x-y)\d y\\
&\geq \kappa \nu(r_E)|E^c\cap B(0, r_E)|+	\int_{E^c\cap B^c(x, r_E)} \nu(x-y)\d y\\
&= \kappa \nu(r_E)|E\cap B^c(0, r_E)|+	\int_{E^c\cap B^c(x, r_E)} \nu(x-y)\d y\\
&\geq \kappa^2 \int_{E\cap B^c(x, r_E)} \hspace{-2ex}\nu(x-y)\d y +	\int_{E^c\cap B^c(x, r_E)} \hspace{-3ex}\nu(x-y)\d y\\
&\geq \kappa^2 \int_{B^c(x, r_E)} \nu(x-y)\d y = \kappa^2 \frac{w^p(|E|)}{|E|}.
\end{align*}
\end{proof}

\noindent We also need the following lemma dealing with convexity of the critical function $\phi$. 
\begin{lemma}\label{lem:gene-convex}
Assume that \ref{item:assump-C} is satisfied. Let $(a_k)_{k\in \mathbb{Z}}$ be a nonincreasing nonnegative sequence, i.e., $0\leq a_{k+1}\leq a_k$, and $T>0$. Then the following estimate holds true 
\begin{align}\label{eq:gene-convex}
\phi_p(\frac{\theta^p}{T})\sum_{k\in \mathbb{Z}} \Big( \frac{1}{\phi^{-1}(1/a_k)}\Big)^p T^{k} \leq \sum_{k\in \mathbb{Z}} \frac{a_{k+1}}{a_k} \Big( \frac{1}{\phi^{-1}(1/a_k)}\Big)^p T^{k}.
\end{align}
\end{lemma}

\begin{proof}
\noindent First, taking $s= \phi^{-1} (1/t')$ and $t= \phi^{-1} (1/s')$ the growth condition in \eqref{eq:growth-condition} becomes 
\begin{align}\label{eq:growth-equiv}
\phi^{-1}_p(\frac{s'}{t'}) w^p(t')\geq \theta^p w^p(s')\qquad\text{for all $s'\leq t'$}. 
\end{align} 
There is no loss of generality if we assume that the right-hand side of \eqref{eq:gene-convex} is finite and that, for $n\geq 1$ sufficiently large, $\lambda_n>0$ with 
$$\lambda_n= 
\sum_{|k|\leq n }w^p(a_k)T^{k}=\sum_{ k\in \mathbb{Z}}  w^p(a'_k)T^{k}$$ 
where here, $a'_k= a_k$ if $|k|\leq n$ and 0 if $|k|>n$. This makes sense as $w(0)=\phi_p(0)=0$. In virtue of the Jensen inequality and the estimate \eqref{eq:growth-equiv} we obtain the following 
\begin{align*}
 \sum_{k\in \mathbb{Z}} 2^{kp} \frac{d_{k+1}}{a_k} \Big( \frac{1}{\phi^{-1}(1/a_k)}\Big)^p&=
\sum_{k\in \mathbb{Z}} \frac{a_{k+1}}{a_k} w^p(a_k)T^{k} 
 \geq \lambda_n \sum_{k\in \mathbb{Z}} \phi_p\big( \phi_p^{-1}\big(\frac{a'_{k+1}}{a_k}\big)\big) \frac{1}{\lambda_n} w^p(a_k) T^{k}\\
&\geq \lambda_n \phi_p\Big( \frac{1}{\lambda_n}\sum_{k\in \mathbb{Z}} \phi_p^{-1}\big(\frac{a'_{k+1}}{a_k}\big) w^p(a_k) T^{k}\Big)\\
& \geq \lambda_n\phi_p\Big( \frac{\theta^p}{\lambda_n T} \sum_{k\in \mathbb{Z}} w^p(a'_{k}) T^{k} \Big)
= \phi_p\big(\frac{\theta^p }{T}\big) \sum_{|k|\leq n}  w^p(a_k) T^{k}.
\end{align*}
\noindent Letting $n\to \infty$ gives the sought inequality since $w(t) = \frac{1}{\phi^{-1} (1/t)}$ as in \eqref{eq:invert-cond}. 
\end{proof}

\vspace{1mm}

\noindent 	In connection with Lemma \ref{lem:gene-convex}, we take $\phi(t)= t^q$, $q\geq 1$ to obtain the following particular result. 
\begin{lemma} \label{lem:young-discrete} For a nonnegative sequence $(a_k)_{k\in \mathbb{Z}}$, $T>0$ and $q\geq 1$ we have 
\begin{align}
\sum_{k\in \mathbb{Z}} a_k^{1/q}T^k\leq T^{q}\sum_{k\in \mathbb{Z}} \frac{a_{k+1}}{a_k}a_k^{1/q}T^k.
\end{align}
\end{lemma}

\begin{proof}
It suffices to assume that $0<\sum_{k\in \mathbb{Z}} a_k^{1/q}T^k<\infty$. Let the counting measure $\d\mu(k) =a_k^{1/q}T^k $ so that $\d\mu(k+1) =T\Big(\frac{a_{k+1}}{a_k}\Big)^{1/q}\d\mu(k) $. Jensen's inequality yields the sought inequality since 
\begin{align*}
1=	\Big(\frac{1}{\mu(\mathbb{Z}) }\int_{\mathbb{Z}}\d \mu(k)\Big )^q=\Big( \frac{1}{\mu(\mathbb{Z}) } \int_{\mathbb{Z}} T\Big(\frac{a_{k+1}}{a_k}\Big)^{1/q} \d \mu(k)\Big )^q &\leq \frac{1}{\mu(\mathbb{Z}) }\int_{\mathbb{Z}} T^q\frac{a_{k+1}}{a_k}\d \mu(k).
\end{align*}
\end{proof}

\noindent Consequently, taking $q= p^*_s/p\geq 1$ in Lemma \ref{lem:young-discrete} results with the following inequality; compare with \cite[Lemma 6.2]{Hitchhiker} or \cite[Lemma 5]{SV11}. 

\begin{corollary}
Assume $\frac{1}{p^*_s} = \frac1p- \frac {s}d>0$,  $s\in(0,1)$. 
For every $T>0$ and every nonnegative sequence $(a_k)_{k\in \mathbb{Z}}$ the following estimate holds
\begin{align*}
\sum_{k\in \mathbb{Z}} a_k^{(d-sp)/d}T^k\leq T^{d/(d-sp)}\sum_{k\in \mathbb{Z}} a_{k+1}a_k^{-sp/d}T^k.
\end{align*}
\end{corollary}
\noindent The next lemma is an immediate consequence of the relation in \eqref{eq:luxemburg-norm} and provides an interplay between a Luxemburg norm associated with a function $\psi$ and that of the mapping $t\mapsto \psi(t^q)$. 
\begin{lemma}\label{lem:psiq-convex}
Let $\psi:[0,\infty]\to[0,\infty]$ be a Young function and $q>0$. Assume $\overline{\psi}_q(t)= \psi(t^q)$ is also a Young function then $u\in L^{\overline{\psi}_q}(\R^d)$ if and only if $u^q\in L^{\psi}(\R^d)$. 
Moreover, we have 
\begin{align*}
\|u\|_{L^{\overline{\psi}_q}(\R^d)}= 	\|u^q\|^{1/q}_{L^{\psi}(\R^d)}. 
\end{align*}
\end{lemma}

\noindent 
To prove Theorem \ref{thm:main-result}, we use the measure theoretic decomposition of functions by level sets. 
\begin{proof}[\textbf{Proof of Theorem \ref{thm:main-result}}]
Without loss of the generality assume that $u\geq 0$ and $|u|_{W^p_\nu(\R^d)}<\infty$. 	For each $k\in \mathbb{Z}$ define 
\begin{alignat*}{2}
A_k&= \big\{ u> 2^k \big\} \quad\,\, \text{and}\quad 
D_k= A_k\setminus A_{k+1}= \big\{2^k< u\leq 2^{k+1} \big\} ,\\
a_k&= \big|\{ u> 2^k \}\big|\quad \text{and}\quad d_k= \big|D_k\big|= a_k-a_{k+1} .
\end{alignat*}
Note that $A_{k+1}\subset A_k$ and hence $a_{k+1}\leq a_k$. 
Moreover, ${D_k}'s$ are disjoints, cover $\R^d$ and we get 
\begin{align}\label{eq:disjoint-union}
A^c_{k+1} = \bigcup_{\ell\leq k} D_\ell \quad \text{and}\quad A_{k} = \bigcup_{\ell\geq k} D_\ell. 
\end{align}
Accordingly we find that 
\begin{align}\label{eq:sum-ak-dk}
a_{k} = \sum_{\ell\geq k} d_\ell\quad \text{and} \quad d_{k} = a_k- \sum_{\ell\geq k+1} d_\ell. 
\end{align}
\noindent Given $x\in D_i$ and $y\in D_j$ with $j\leq i-2$, we have $|u(x)-u(y)|\geq 2^i-2^{j+1}\geq 2^{i-1}.$ Therefore, according to \eqref{eq:disjoint-union} and Lemma \ref{lem:inequality-indicator}, one deduces the following
\begin{align*}
\iil_{\R^d\R^d} |u(x)-u(y)|^p\nu(x-y)\d y\d x
&= \sum_{i\in \mathbb{Z}} \sum_{ j\in \mathbb{Z}}~\iil_{D_iD_j} |u(x)-u(y)|^p\nu(x-y)\d y\d x\\
&\geq 2\sum_{i\in \mathbb{Z}} \sum_{\mathop{ }^{~ j\in \mathbb{Z}}_{ i\leq j-2}}~\iil_{D_iD_j} 2^{p(i-1)}\nu(x-y)\d y\d x\\
&= 2\sum_{ i\in \mathbb{Z}} 2^{p(i-1)} \int_{D_i} \int_{A^c_{i-1}} \nu(x-y)\d y\d x\\
&\geq 2\kappa^2\sum_{i\in \mathbb{Z}} 2^{p(i-1)} d_i\frac{w^p(a_{i-1})}{a_{i-1}} . 
\end{align*}
In short, using the relation \eqref{eq:invert-cond} we have 
\begin{align}\label{eq:semi-norm1}
|u|^p_{W^p_\nu(\R^d) }& \geq
2\kappa^2\sum_{k\in \mathbb{Z}} 2^{pk} \frac{d_{k+1}}{a_k} \Big( \frac{1}{\phi^{-1}(1/a_k)}\Big)^p.
\end{align}
\noindent Recalling \eqref{eq:sum-ak-dk} and that $d_i=a_i-a_{i+1}$, we get
\begin{align}\label{eq:semi-norm2}
|u|^p_{W_\nu^p(\R^d)} & \geq 2\kappa^2\sum_{\mathop{ }^{~ i\in \mathbb{Z}}_{ a_{i-1} \neq 0}} 2^{p(i-1)} a_i\frac{w^p(a_{i-1})}{a_{i-1}} - 2\kappa^2 \sum_{\mathop{ }^{~ i\in \mathbb{Z}}_{ a_{i-1} \neq 0}} \sum_{\mathop{ }^{~ i\in \mathbb{Z}}_{ j \geq i+1}} 2^{p(i-1)} d_j\frac{w^p(a_{i-1})}{a_{i-1}}. 
\end{align}

\noindent Since  $d_j\leq a_j$,  using Fubini's theorem and the formula  $$\sum_{\mathop{ }^{~ i\in \mathbb{Z}}_{ i \leq j-1}} c^{i-1} = \frac{c^{j-1}}{c-1} \qquad\text{for $c>1$,}$$
we find that 
\begin{align*}
\sum_{\mathop{ }^{~ i\in \mathbb{Z}}_{ a_{i-1} \neq 0} } \sum_{\mathop{ }^{~ j\in \mathbb{Z}}_{  i+1\leq j}} 2^{p(i-1)} d_j\frac{w^p(a_{i-1})}{a_{i-1}}
&=\sum_{\mathop{ }^{~ j\in \mathbb{Z}}_{ a_{j-1} \neq 0}} \sum_{\mathop{ }^{~ i\in \mathbb{Z}}_{ i
\leq j-1}} 2^{p(i-1)} d_j\frac{w^p(a_{i-1})}{a_{i-1}}\\
&\leq  \frac{1}{2^p-1}\sum_{\mathop{ }^{~ j\in \mathbb{Z}}_{ a_{j-1} \neq 0}} 2^{p(j-1)} a_j\frac{w^p(a_{j-1})}{a_{j-1}} .
\end{align*}
Combining this together with \eqref{eq:semi-norm1} and \eqref{eq:semi-norm2}, recalling $C_p(t)= 2\kappa^2\frac{t^p-2}{t^p-1}$, $t>1$ yields 
\begin{align}\label{eq:semi-norm}
|u|^p_{W^p_\nu(\R^d)} & \geq C_p(2) \sum_{k\in \mathbb{Z}} 2^{pk} \frac{a_{k+1}}{a_k}w^p(a_k)=
C_p(2) \sum_{k\in \mathbb{Z}} 2^{pk} \frac{a_{k+1}}{a_k} \Big( \frac{1}{\phi^{-1}(1/a_k)}\Big)^p
\end{align}
\noindent In order to employ  \ref{item:assump-C}, we emphasize that  $d_{k+1}\leq a_{k+1}\leq a_k$.  In virtue of Lemma \ref{lem:gene-convex} with $T=2^p$ and the fact that $\phi^{-1} $ is nondecreasing, the following estimates hold true
\begin{align}\label{eq:young-ine}
\sum_{k\in \mathbb{Z}} 2^{pk} \frac{a_{k+1}}{a_k} \Big( \frac{1}{\phi^{-1}(1/a_k)}\Big)^p
\geq \phi(\frac{\theta}{2})\sum_{k\in \mathbb{Z}} 2^{pk} \Big( \frac{1}{\phi^{-1}(1/a_k)}\Big)^p\geq \phi(\frac{\theta}{2})\sum_{k\in \mathbb{Z}} 2^{pk} \Big( \frac{1}{\phi^{-1}(1/d_k)}\Big)^p.
\end{align}

\noindent Finally, since $u^p=\sum_{ k\in \mathbb{Z}} u^p\mathds{1}_{D_k}$, in view of Lemma \ref{lem:psiq-convex} and the relation \eqref{eq:luxem-indicator} we find that 
\begin{align}\label{eq:dyadic-orlicz}
\|u\|^p_{L^\phi(\R^d)} &= \|u^p\|_{L^{\phi_p}(\R^d)} 
\leq \sum_{ k\in \mathbb{Z}} \|u^p\mathds{1}_{D_k}\|_{L^{\phi_p}(\R^d)} 
\leq2^p \sum_{ k\in \mathbb{Z}} 2^{pk} \Big(\frac{1}{\phi^{-1}(1/d_k)}\Big)^p.
\end{align}
\noindent Merging together \eqref{eq:semi-norm}, \eqref{eq:young-ine} and \eqref{eq:dyadic-orlicz} gives the desired inequality for $t=2$ 
\begin{align*}
|u|^p_{W^p_\nu(\R^d) }\geq \frac{1}{\Theta_2^p} \|u\|^p_{L^\phi(\R^d)}, \quad \Theta_2^p=  2^{p-1} [\kappa^2C_p(2)\phi(\frac{\theta}{2})]^{-1} .
\end{align*}
\noindent More generally, using the level sets decomposition $A_k(t)= \{u>t^k\}$ and $D_k(t)= \{t^k<u\leq t^{k+1}\}$, for $t\geq 2$, in place of $A_k$ and $D_k$ and repeating with a close look at the proof of the previous case, t=2, reveals the desired estimate \eqref{eq:main-equality}.
It immediately follows that the embedding $W^{p}_\nu(\R^d) \hookrightarrow L^\phi(\R^d)$ is continuous while the continuity of the embedding $W^{1,p}(\R^d) \hookrightarrow L^\phi(\R^d)$ stems from the continuous embedding $W^{1,p}(\R^d) \hookrightarrow W_\nu^p(\R^d)$; see the relation \eqref{eq:equiv-sobolev}. This ends the proof of Theorem \ref{thm:main-result}. 
\end{proof}

\noindent
We recall that by Example \ref{Ex:fractional}, the fractional Gagliardo-Nirenberg-Sobolev inequality \eqref{eq:frac-sobolev-inequality} is a direct consequence of Theorem \ref{thm:main-result}, where, $\nu(h)= |h|^{-d-sp}$ is associated with the critical Young function $\phi(t) =\gamma_s^{p^*_s/p} t^{p^*_s}$, wherein, the growth factor is given by $\theta= 1/\gamma_s^{1/p}$. However, for the convenience of the reader, we provide a simple proof credited to Haim Brezis and  incorporated in \cite[Proposition 15.5]{Ponce16elliptic}.
\begin{theorem}[Fractional Sobolev inequality]\label{thm:gagliardo-sobolev}
Let $s \in (0, 1)$ such that $\frac{1}{p^*_s} = \frac1p- \frac {s}d>0$ then 
\begin{align*}
\|u\|_{L^{p^*_s}(\R^d)}\leq 2^{p^*_s/p} |B(0,1)|^{-1/p-s/d}\Big(\iil_{\R^d\R^d} \frac{|u(x)- u(y)|^p}{|x-y|^{d+sp}}\d x\d y\Big)^{1/p}\quad\text{for all} \quad u\in L^{p^*_s}(\R^d). 
\end{align*} 
\end{theorem}
\begin{proof} 
Fix $ x\in \R^d$ and $r>0$. Integrating the inequality $|u(x)|\leq |u(y)|+ |u(x)-u(y)|$ over $y\in B(x,r)$ and using Jensen's inequality implies 
\begin{align*}
|u(x)|&\leq \fint_{B(x,r)}|u(y)|\d y+ \fint_{B(x,r)}|u(x)-u(y)|\d y\\
&\leq \Big(\fint_{B(x,r)}|u(y)|^{p^*_s}\d y\Big)^{1/p^*_s}+ \Big(\fint_{B(x,r)}|u(x)-u(y)|^p\d y\Big)^{1/p}\\
&\leq \Big(\fint_{B(x,r)}|u(y)|^{p^*_s}\d y\Big)^{1/p^*_s}+ \Big(r^{d+sp}\fint_{B(x,r)}\frac{|u(x)-u(y)|^p}{|x-y|^{d+sp}}\d y\Big)^{1/p}\\
&\leq r^{-d/p^*_s}|B(0,1)|^{-1/p^*_s}\Big(\int_{\R^d}|u(y)|^{p^*_s}\d y\Big)^{1/p^*_s}+ r^s|B(0,1)|^{-1/p}\Big(\int_{\R^d}\frac{|u(x)-u(y)|^p}{|x-y|^{d+sp}}\d y\Big)^{1/p}.
\end{align*}
\noindent Now we choose $r$ such that both summands of the last inequality are equal. To more  be precise,
\begin{align*}
r(x)= r=|B(0,1)|^{1/d-p/dp^*_s}\Big(\int_{\R^d}|u(y)|^{p^*_s}\d y\Big)^{p/dp^*_s}\Big(\int_{\R^d}\frac{|u(x)-u(y)|^p}{|x-y|^{d+sp}}\d y\Big)^{-1/d}.
\end{align*} 
Substituting this specific $r(x)$ in the preceding estimate leads to 
\begin{align*}
|u(x)|^{p^*_s}
&\leq 2^{p^*_s} r^{-d}(x)|B(0,1)|^{-1}\Big(\int_{\R^d}|u(y)|^{p^*_s}\d y\Big)\\
&= 2^{p^*_s} |B(0,1)|^{-2+p/p^*_s}\Big(\int_{\R^d}|u(y)|^{p^*_s}\d y\Big)^{1-p/p^*_s}\Big(\il_{\R^d} \frac{|u(x)-u(y)|^p}{|x-y|^{d+sp}}\d y\Big). 
\end{align*}
\noindent This implies an equivalent of the sought inequality, as integrating with respect to $x$ yields, 
\begin{align*}
\int_{\R^d}|u(x)|^{p^*_s}\d x
&\leq 2^{p^*_s} |B(0,1)|^{-1-sp/d}\Big(\int_{\R^d}|u(y)|^{p^*_s}\d y\Big)^{1-p/p^*_s}\Big(\iil_{\R^d\R^d} \frac{|u(x)-u(y)|^p}{|x-y|^{d+sp}}\d y\d x\Big). 
\end{align*}
\end{proof}
\noindent The next result is a sort of converse to Example \ref{Ex:fractional}. It implies that given a radial kernel $\nu$, 
then $\phi(t) = ct^q$ if and only if $\nu$ is a fractional kernel of the form $C|h|^{-d-sp}$, $s\in (0,1).$ 
\begin{theorem}\label{thm:fractional} Let $\nu:\R^d\setminus\{0\}\to [0,\infty)$ be radial.  Let \ref{item:assump-A}, \ref{item:assump-B} and \ref{item:assump-C} be in force. Assume that $\nu$ is associated with the critical function $\phi(t)=c t^q$ for some $q,c>0$. Then necessarily $\frac{1}{p}-\frac{1}{d}<\frac{1}{q}<\frac{1}{p}$ and there exists $s\in (0,1)$, in fact, $s=\frac{d}{p}-\frac{d}{q}$, such that $\nu(h)= C_{p,q,d}|h|^{-d- sp}$ for some constant $C_{p,q,d}>0$ depending on $c,p,q$ and $d$.
\end{theorem}

\begin{proof}
First of all, in virtue of \ref{item:assump-C}, observe that $\phi_p(t)= ct^{q/p}$ with $c>0$ is convex if and only if $q\geq p\geq 1$. The relation \eqref{eq:invert-cond} implies that $w(r)=c^{1/q}r^{1/q}$ and hence \eqref{eq:defw1} amounts to
\begin{align*}
c^{p/q}r^{p/q-1}= \int_{B^c(0, \eta(r))} \nu(h)\d h
= dc_d\int_{\eta(r)}^\infty \nu(\tau)\tau^{d-1}\d \tau = \int_{r}^\infty \nu(\eta(\tau'))\d \tau'. 
\end{align*}
Differentiating both sides and letting $\rho= \eta(r)= \big(\frac{r}{c_d}\big)^{1/d}$ yields
\begin{align*}
\nu(\rho)= c^{p/q}(1-\frac{p}{q}) c_d^{p/q-2}\, \rho^{-d+ dp/q-d} =C_{p,q,d} \,\rho^{-d-sp}\quad\text{ with $s= \frac{d}{p}-\frac{d}{q}\in [0, d]$}. 
\end{align*}
In short, $\nu(h)= C_{p,q,d} \,|h|^{-d-sp}.$ Finally observe that, by \ref{item:assump-A}, $\nu\in L^1(\R^d, 1\land |h|^p)$ if and only if $s\in (0, 1)$ that is $\frac{1}{p}-\frac{1}{d}<\frac{1}{q}$ and $q>p$. This ends the proof. 
\end{proof}

\bigskip
\vspace{2mm}

\noindent\textbf{Symmetric rearrangement:} In view of generalizing Theorem \ref{thm:main-result} we need to bring into play the symmetric rearrangement. We want to weaken the assumptions on $\nu$, by possibly enlarging them to the class of non-radial kernels. Ultimately, we recall some essential notions of symmetric decreasing rearrangement; see for instance \cite{Bae19, Gra14,LL01} for more details. Let $E\subset \R^d$ be a measurable set with $|E|<\infty$. The symmetric rearrangement of $E$ denoted $E^*=B(0, r_E)$ is the open ball having the same volume with $E$, i.e., $r_E= \eta(|E|)= \big(\frac{|E|}{c_d}\big)^{1/d}$. The symmetric rearrangement of a measurable function $u:\R^d\to\R$, is the function denoted $u^*:\R^d\to [0,\infty)$ and defined by 
\begin{align}\label{eq:sym-rearran}
u^*(x)= u^*(|x|)= \int_0^\infty \mathds{1}_{\{ |u| >s\}^* }(x)\d s= \inf\{s>0: ~|\{ |u| >s\}|\leq c_d|x|^d\}. 
\end{align}

\noindent It is a routine to check that the identity in \eqref{eq:sym-rearran} holds true. Obviously the function $u^*$ is radial and radially nonincreasing, i.e., $u^*(x)\leq u^*(y)$ whenever $|x|\geq |y|$. Furthermore, $\{ |u|>s\}^*= \{ u^* >s\},$ for all $s>0.$ This implies that $u^*$ and $u$ are equimeasurables, i.e., $ |\{ |u| >s\}|= |\{ u^* >s\}|$ for all $s>0$ and that $u^*$ is lower semi-continuous. Next, assume $u^*(r)<\infty$, $r>0$ and let $s_n=u^*(r)-\frac{1}{n}$, for $n\geq 1$ large. The inf characterization in \eqref{eq:sym-rearran} implies $s_n\not\in \{s>0: |\{ |u| >s\}|\leq c_dr^d\} $. Thus, $|\{ |u| >u^*(r)-\frac{1}{n} \}|>c_dr^d$ and hence letting $n\to \infty$ we get 
\begin{align*}
|\{ |u| \geq u^*(r) \}|\geq c_dr^d= |B(0,r)|.
\end{align*}
\noindent Since $u^*$ is radially nonincreasing, $\{ u^* >u^*(r) \}\subset B(0, r)\subset \{ u^* \geq u^*(r) \}\subset \overline{B}(0,r)$ . Hence we get 
\begin{align*}
|\{ |u| >u^*(r) \}|= |\{ u^* >u^*(r) \}|\leq |B(0, r)| = |\{ u^* \geq u^*(r) \}|.
\end{align*}
This combined with the previous inequality yields
\begin{align}\label{eq:equime-ineq}
|\{ |u| >u^*(r) \}|\leq |B(0, r)| \leq |\{ |u| \geq u^*(r) \}|.
\end{align}
\noindent The equalities hold if $u^*$ is decreasing. The next result, compare with \cite[Lemma 3.1]{JW19}, is a fair generalization of Lemma \ref{lem:inequality-indicator}. 
\begin{theorem}\label{thm:inequality-indicator}
Let $\nu: \R^d\setminus \{0\}\to [0,\infty)$ be measurable and note $\nu^*$ be its symmetric rearrangement.
Let $E\subset \R^d$ be measurable such that $|E|<\infty$, define $ r_E= \eta(|E|)= \big(\frac{|E|}{c_d}\big)^{1/d}$, then 
\begin{align*}
&\essinf_{x\in\R^d}\int_{E^c} \nu(x-y)\d y \geq \nu^\# (|E|), \quad\text{with} \quad \nu^\# (|E|)= \int_{ \{\nu <\nu^*(r_E)\} }\nu(h)\d h,\\
&\esssup_{x\in\R^d}\int_{E} \nu(x-y)\d y \leq \nu_\# (|E|), \quad\text{with} \quad \nu_\# (|E|)= \int_{ \{\nu \geq \nu^*(r_E)\} }\nu(h)\d h.
\end{align*}	
\noindent Moreover, $\nu^\# (|E|)\to \int_{\R^d} \nu(h)\d h$ 
as $|E|\to0$. If  $\nu \in L^1(\R^d\setminus B(0, \delta ))$ for some $\delta>0$ then $\nu^\# (|E|)\to 0$ 
as $|E|\to\infty$. 
\end{theorem}

\begin{proof}
We only prove the first inequality and the second follows analogously.	Fixing $x\in \R^d$, it is sufficient to assume that $\int_{E^c} \nu(x-y)\d y <\infty$. Let $E_x= x+E$ so that $|E_x|=|E|$ and hence  we get 
\begin{align*}
\int_{E^c} \nu(x-y)\d y &= \int_{E^c_x} \nu(h)\d h= \Big[\int_{\{\nu <\nu^*(r_E)\} } - \int_{E_x\cap \{\nu <\nu^*(r_E)\} }+ \int_{E^c_x\cap \{\nu \geq \nu^*(r_E)\}} \Big] \nu(h)\d h\\
&\geq \int_{\{\nu <\nu^*(r_E)\} } \nu(h)\d h - \nu^*(r_E) |E_x\cap \{\nu <\nu^*(r_E)\} |+ \nu^*(r_E) |E^c_x\cap \{\nu \geq \nu^*(r_E)\} |\\
&= \int_{\{\nu <\nu^*(r_E)\} } \nu(h)\d h + \nu^*(r_E) \big( |E^c_x\cap \{\nu \geq \nu^*(r_E)\} |- |E_x\cap \{\nu < \nu^*(r_E)\} |\big)\\
& \geq \int_{\{\nu <\nu^*(r_E)\} } \nu(h)\d h. 
\end{align*}

\noindent The last inequality follows since, as inequality \eqref{eq:equime-ineq} implies $|\{\nu \geq \nu^*(r_E)\}|\geq |E| $, we have 
\begin{align*}
|E^c_x\cap \{\nu \geq \nu^*(r_E)\} | &- |E_x\cap \{\nu < \nu^*(r_E)\} |\\
&= (|\{\nu \geq \nu^*(r_E)\}|- |E_x\cap \{\nu \geq \nu^*(r_E)\} |) - (|E|- |E_x\cap \{\nu \geq \nu^*(r_E)\} | ) \\
&= |\{\nu \geq \nu^*(r_E)\}|- |E|\geq |B(0,r_E)|-|E|=0. 
\end{align*}
Meanwhile, $\nu^* (0)= \inf\{s>0: ~|\{ \nu >s\}|=0\}=  \|\nu\|_{L^\infty(\R^d)}$, hence  we get $\nu^\# (|E|)\to \int_{\R^d} \nu(h)\d h$ as $|E|\to0$. If $\nu \in L^1(\R^d\setminus B(0, \delta ))$, then a convergence argument implies $\nu^\# (|E|)\to 0$ as $|E|\to\infty$. 
\end{proof}

\noindent We mention in passing that Lemma \ref{lem:inequality-indicator} and Theorem \ref{thm:inequality-indicator} generalize \cite[Lemma 6.1]{Hitchhiker} focusing on the particular kernel, $\nu(h)= |h|^{-d-sp}$ for $ s\in (0,1)$. Indeed, the latter case follows by observing that  $\nu=\nu^*$. More generally, there holds the  following.
\begin{corollary} Let $E\subset \R^d$. Let $\gamma'_\vartheta
= \frac{d}{\vartheta} c_d^{1+\frac{\vartheta}{d}}$, $\vartheta>0$.  For $\beta'>0 $, $0<\beta\leq  d$ we have
\begin{align*}
\essinf_{x\in\R^d}	\int_{E^c} \frac{\d y}{|x-y|^{d+\beta'}} \geq \gamma'_{\beta'}|E|^{-\beta'/d}
\quad\text{and }
\quad \esssup_{x\in\R^d}\int_{E} \frac{\d y}{|x-y|^{d-\beta }} \leq \gamma'_\beta|E|^{\beta/d}. 
\end{align*}
\end{corollary}

\medskip 
\noindent We are now ready to state a refined version of Theorem \ref{thm:main-result} under weaker 
assumptions. For this, we need to redefine the potential $w$ in \eqref{eq:defw1} as follows  	
\begin{align}\label{eq:defw2}
w(r) =  (r\nu^\#(r))^{1/p}= \Big(|B(0,\eta(r))| \int_{\{\nu< \nu^*(\eta(r)) \}} \nu(h) \d h \Big)^{1/p}, \qquad \eta(r)= \big(\frac{r}{c_d}\big)^{1/d}. 
\end{align}
We emphasize that $\nu^*$ is the symmetric rearrangement of $\nu$. 
\begin{theorem}\label{thm:main-result2}
Let $\nu: \R^d\setminus \{0\}\to [0,\infty)$ be measurable. Let $w: [0,\infty)\to [0,\infty)$ with $w(r)= (r\nu^\#(r))^{1/p}$ be as in \eqref{eq:defw2}. Assume that: $r\mapsto w(r)$ is invertible,  $t\mapsto \phi_p(t)= \phi(t^{1/p})$ is a Young function where $\phi(t)= 1/w^{-1}(1/t)$ and that there is $\theta >0$ such that 
\begin{align*}
\phi\big(\theta \frac{s}{t}\big)	\leq \frac{\phi(s)}{\phi(t)} \qquad\text{for all $\,\,0\leq s\leq t$}.
\end{align*}
\noindent Define $\Theta_t=  t[2C_p(t)\phi(\frac{\theta}{t})]^{-1/p}$ with $C_p(t)=\frac{t^p-2}{t^p-1}$, $t\geq 2$. Then the following inequality holds
\begin{align*}
\|u\|_{L^\phi(\R^d)} \leq \Theta_t \Big(\iil_{\R^d\R^d} |u(x)-u(y)|^p\nu(x-y)\d y\d x\Big)^{1/p}\quad\text{for all}\quad u \in L^\phi(\R^d).
\end{align*}
\end{theorem} 

\begin{proof}
The proof follows exactly the lines of the proof of Theorem \ref{thm:main-result}, as the only major change is the analog of the estimate \eqref{eq:semi-norm1} which easily derives from Theorem \ref{thm:inequality-indicator}. 
\end{proof}

\vspace{1mm}
\begin{remark}
The assertions $(i)-(vii)$ of Proposition \ref{prop:growth-phi} remain true for a general kernel $\nu$ such that 
$ w(r)= (r\nu^{\#}(r))^{1/p}$ and $\phi(t) = 1/w^{-1}(1/t)$ exist. 
\end{remark}

\noindent We now present two observations for which, it is possible to eschew the lack of certain assumptions of Theorem \ref{thm:main-result2} and retrieve a similar conclusion. The first observation implies that lack of growth condition \eqref{eq:growth-condition} may sometime not be a direct obstacle. For instance we saw in Example \ref{Ex:fractional-max} our inequality remains true although $\phi: t\mapsto \max(t^{p^*_{s_1}}, t^{p^*_{s_2}})$, $p^*_{s_2}>p^*_{s_1}> p$ does not satisfy the growth condition \eqref{eq:growth-condition}, but this is compensated by the fact that each $t\mapsto t^{p^*_{s_i}}, i=1,2$ verifies the growth condition \eqref{eq:growth-condition}.  The second observation implies that one can compensate the lack of convexity. For instance, as we mentioned in  Example \ref{Ex:fractionalmin} our inequality remains true although $\phi: t\mapsto \min(t^{p^*_{s_1}}, t^{p^*_{s_2}})$, $p^*_{s_2}>p^*_{s_1}>p$ is not convex but the inequality \eqref{eq:main-equality} remains true. The lack of convexity is compensated by the minorant convex. This is the goal of the next result.
\begin{theorem}\label{thm:main-partial}
Let $\nu: \R^d\setminus \{0\}\to [0,\infty)$ be measurable and $\phi$ be the corresponding critical function. 
\noindent Assume there exist  $\phi_i$, $i=1,2$ such that $t\mapsto\phi_i(t^{1/p})$ is convex and satisfies the growth condition: $\phi_i(t)\phi_i\big(\theta_i\frac{s}{t}\big)\leq \phi_i(s)$ for all $s\leq t$. Consider the following two scenarios.

\begin{enumerate}[$(i)$]
\item We have $\phi(t)= \max(\phi_1(t), \phi_2(t)) $ and there exist $\nu_i: \R^d\setminus \{0\}\to [0,\infty)$, $i=1,2$ associated with $\phi_i$ such that $c^{-1}\nu(h)\leq \max(\nu_1(h), \nu_2(h))\leq c\nu(h)$ for some $c\geq 1$. 
\item  We have $\phi(t)= \min(\phi_1(t), \phi_2(t))$ which is identified with $\phi_{\min}$, 
\begin{align*}
\phi_{\min}(t)=\int_0^t\frac{\min(\phi_1(s), \phi_2(s))}{s}\d s. 
\end{align*}
\end{enumerate} 
Then in either scenario $(i)$ and $(ii)$ the following inequality remains valid
\begin{align*}
\|u\|_{L^\phi(\R^d)} \leq C\Big(\iil_{\R^d\R^d} |u(x)-u(y)|^p\nu(x-y)\d y\d x\Big)^{1/p}\quad\text{for all}\quad u \in L^\phi(\R^d). 
\end{align*}

\end{theorem}

\begin{proof}
$(i)$ Assume $\phi(t)= \max(\phi_1(t), \phi_2(t))$ so that $\|u\|_{ L^{\phi} (\R^d) }
\leq 2\max(\|u\|_{L^{\phi_1} (\R^d)},\|u\|_{L^{\phi_2} (\R^d)})$, by \eqref{eq:orlicz-inter-max}. Since Theorem \ref{thm:main-result2} holds  for each couple $(\phi_i,\nu_i),\,i=1,2$, the desired inequality follows. Indeed, for $u\in L^\phi(\R^d) = L^{\phi_1}(\R^d)\cap L^{\phi_2}(\R^d)$, 
\begin{align*}
\begin{split}
\|u\|_{L^{\phi_i} (\R^d)}&\leq C_i\Big(\iil_{\R^d\R^d} |u(x)-u(y)|^p\nu_i(x-y)\d y\d x\Big)^{1/p}\\
&\leq C\Big(\iil_{\R^d\R^d} |u(x)-u(y)|^p\nu(x-y)\d y\d x\Big)^{1/p}, \,\, C=\max(C_1, C_2)c_2^{1/p}.
\end{split}
\end{align*}

\noindent $(ii)$ Assume $\phi(t)= \min(\phi_1(t), \phi_2(t))$. In virtue of Proposition \ref{prop:growth-phi}  $t\mapsto \frac{\phi_i(t^{1/p})}{t}$, $i=1,2$ is increasing, it follows that  $t\mapsto \phi_{\min} (t^{1/p}) $ is convex since 
\begin{align*}
\phi_{\min}(t^{1/p})=\frac1p \int_0^t\frac{\min(\phi_1(s^{1/p}), \phi_2(s^{1/p}))}{s}\d s. 
\end{align*}

\noindent Moreover, putting $\theta'=\min(\theta_1, \theta_2)$ one easily checks that
\begin{align*}
& \phi\big(\theta'\frac{s}{t}\big)	\leq \frac{\phi(s)}{\phi(t)} \quad\text{for all $s\leq t$ \quad and }\quad \phi_{\min}(t)\leq \phi(t)\leq\phi_{\min}(2t)\quad\text{for all $t> 0$}. 
\end{align*}

\noindent Altogether, this implies that $\phi_{\min}$ satisfies the growth condition \eqref{eq:growth-condition} with $\theta=\frac{1}{2}\min(\theta_1, \theta_2)$. Indeed, 
\begin{align*}
\phi_{\min}\big(\frac{\theta' s}{2t}\big)\leq \phi\big(\frac{\theta' s}{2t}\big) \leq \frac{\phi(s/2)}{\phi(t)}\leq \frac{\phi_{\min}(s)}{\phi_{\min}(t)} \quad\text{for $s\leq t$.} 
\end{align*} 
\noindent It turns out that Lemma \ref{lem:gene-convex} applies to $\phi_{\min}$ and hence, since $\phi^{-1}(1/r)\leq\phi^{-1}_{\min}(1/r)\leq 2\phi^{-1}(1/r)$, as a substitute for the inequality\eqref{eq:gene-convex} one readily obtains that
\begin{align}\tag{\ref{eq:gene-convex}'}\label{eq:gene-convex-bis}
\phi_{\min}(\frac{\theta}{T^{1/p}})\sum_{k\in \mathbb{Z}} \Big( \frac{1}{\phi^{-1}(1/d_k)}\Big)^p T^{k} \leq 2^p \sum_{k\in \mathbb{Z}} \frac{d_{k+1}}{a_k} \Big( \frac{1}{\phi^{-1}(1/a_k)}\Big)^p T^{k}.
\end{align}
\noindent Whence, a mere adaptation of the proof of Theorem \ref{thm:main-result} provides the desired	 inequality. 
\end{proof}
\noindent An immediate consequence of Theorem \ref{thm:main-result2} is given by the following embeddings. 
\begin{corollary} \label{cor:emb-sobolev} Assume the assumptions of Theorem \ref{thm:main-result2} are in force, with $\nu\in L^1(\R^d, 1\land |h|^p)$. Let $\psi $ be a Young function such that $\psi (ct)\leq \max(t^p, \phi(t))$ for all $t>0$ and for some constant $c>0$. The embeddings $W_\nu^p(\R^d) \hookrightarrow L^\psi(\R^d)$ and $W^{1,p}(\R^d) \hookrightarrow L^\psi(\R^d)$ are continuous.
\end{corollary}

\begin{proof}
Clearly, Theorem \ref{thm:main-result2} implies $W_\nu^p(\R^d) \hookrightarrow L^\phi(\R^d)$ and we naturally have $W_\nu^p(\R^d) \hookrightarrow L^p(\R^d)$.
Hence $W_\nu^p(\R^d) \hookrightarrow L^p(\R^d)\cap L^\phi(\R^d)= L^{\max(t^p, \phi(t)}(\R^d)$, by Theorem \ref{thm:orlicz-emb}. 
Therefore, since $\psi (ct)\leq \max(t^p, \phi(t)),$ Theorem \ref{thm:orlicz-emb} implies that $W_\nu^p(\R^d) \hookrightarrow L^\psi(\R^d)$ and hence $W^{1,p}(\R^d) \hookrightarrow W_\nu^p(\R^d) \hookrightarrow L^\psi(\R^d)$ see the relation \eqref{eq:equiv-sobolev}.
\end{proof}

\noindent Since $q\in [p, p^*_s]$ if and only if $t^q\leq \max(t^p, t^{p^*_s})$ we deduce the following. 
\begin{corollary}\label{cor:emb-sobolev-frac1} If $p^*_s>0,$ $s\in (0, 1)$,  then for every $ q\in [p, p^*_s]$ the embedding $W^{s,p}(\R^d) \hookrightarrow L^{q}(\R^d)$ is continuous. 
\end{corollary}

\noindent More generally, in order to capture the above embeddings in Corollary \ref{cor:emb-sobolev} on arbitrarily open sets, we need to introduce extension domain with respect to the kernel $\nu$.
\begin{definition} An open set $\Omega\subset \mathbb{R}^d$ will be called an $W^p_\nu$-extension domain if there exist a linear operator $E:W^p_\nu(\Omega)\to W^p_\nu (\mathbb{R}^d)$ and a constant $C: = C(\Omega, \nu, d,p)$ 
such that 
\begin{align*}
Eu\mid_{\Omega} &= u \qquad\hbox{and} \qquad \|Eu\|_{W^{p}_\nu(\mathbb{R}^d)}\leq C \|u\|_{W^{p}_\nu(\Omega)}, \quad\text{for all}\quad u \in W^{p}_\nu(\Omega). 
\end{align*}
\end{definition} 

\noindent According to \cite{Zh15}, an open set $\Omega\subset \mathbb{R}^d$ is an $W^{s,p}$-extension domain, $s\in (0,1)$ if and only if $\Omega$ satisfies the measure density condition, i.e., there is $c>0$ such that for every $x\in \Omega$ and $r>0$ we have $|\Omega\cap B(x,r)|>cr^d$. For some authors this condition also means that $\Omega$ is a $d$-set. Note that every Lipschitz domain is an $W^{s,p}$-extension domain even for $s=1$; see \cite{HKT08}. 

\vspace{1mm}
\noindent  If $\nu$ is radially almost decreasing then a bounded bi-Lipschitz domain $\Omega\subset \mathbb{R}^d$ is an $W^p_\nu$-extension domain; see \cite[Theorem 3.78]{Fog20}. For instance, in this context, $\Omega=\R^d$, the half space $\Omega=\R^d_+$ and Balls $\Omega= B(a,r)$, $r>0, a\in \R^d$ are  $W^p_\nu$-extension domains.

\begin{corollary} \label{cor:emb-sobolev-bis} Let the assumptions of Theorem \ref{thm:main-result2} be in force. Assume $\Omega\subset \mathbb{R}^d$ is an $W^p_\nu$-extension domain. Let $\psi $ be a Young function such that $\psi (ct)\leq \max(t^p, \phi(t))$ for all $t>0$ and for some constant $c>0$. 
Then the embedding $W_\nu^p(\Omega) \hookrightarrow L^\psi(\Omega)$ is continuous. Moreover, if $|\Omega|<\infty$ then $W_\nu^p(\Omega) \hookrightarrow L^\psi(\Omega)$ is continuous when 
$\psi (ct)\leq \max(t, \phi(t))$ for all $t\geq 1$. 
\end{corollary}

\begin{proof}
\noindent Clearly, Theorem \ref{thm:main-result2} and the extension property of $\Omega$ imply $W_\nu^p(\Omega) \hookrightarrow L^\phi(\Omega)$ and we naturally have $W_\nu^p(\Omega) \hookrightarrow L^p(\Omega)$. Hence 
$W_\nu^p(\Omega) \hookrightarrow L^p(\Omega)\cap L^\phi(\Omega)= L^{\max(t^p, \phi(t))}(\Omega)$,
where the equality is a consequence of the equivalence \eqref{eq:orlicz-inter-max}. 
This, combined with the fact that $\psi (ct)\leq \max(t^p, \phi(t))$ and Theorem  \ref{thm:orlicz-emb} imply  that $W_\nu^p(\Omega) \hookrightarrow L^\psi(\Omega)$. 

\vspace{1mm}

\noindent Analogously, if $|\Omega|<\infty$ then $L^p(\Omega)\hookrightarrow L^1(\Omega)$ so that we have $W_\nu^p(\Omega) \hookrightarrow L^1(\Omega)\cap L^\phi(\Omega)= L^{\max(t,\phi(t))}(\Omega)$ and since $\psi (ct)\leq \max(t^p, \phi(t))$, $t\geq1$ by Theorem \ref{thm:orlicz-emb} we get $W_\nu^p(\Omega) \hookrightarrow L^{\psi}(\Omega)$. 
\end{proof}

\vspace{1mm}

\noindent Note that $\psi (ct)\leq \max(t, \phi(t))$ for all $t\geq 1$ implies that $\psi (ct)\leq \max(t^p, \phi(t))$ for all $t\geq 1$. Specializing Corollary  \ref{cor:emb-sobolev-bis} to the particular case $\nu(h)= |h|^{-d-sp}$ yields the following well-known embedding result. 
\begin{corollary}\label{cor:emb-sobolev-frac}  Assume $\Omega\subset \mathbb{R}^d$ is an $W^{s,p}$-extension domain, $s\in (0,1)$.  If $p^*_s>0$ then 
$W^{s,p}(\Omega) \hookrightarrow L^{q}(\Omega)$ is continuous for every $q\in [p, p^*_s]$. Moreover, $|\Omega|<\infty$ then 	$W^{s,p}(\Omega) \hookrightarrow L^{q}(\Omega)$ is continuous for every $q\in [1, p^*_s]$. 
\end{corollary}
\begin{proof}
It suffices to observe that $q\in [p, p^*_s]$ if and only if $t^q\leq \max(t^p, t^{p^*_s})$ for all $t>0$.  The case $|\Omega|<\infty$ follows analogously since  $ q\in [1, p^*_s]$ implies $t^q\leq \max(t^p, t^{p^*_s})$ for all $t\geq1$.
\end{proof}

\noindent A natural question arises concerning the compactness of the embeddings. We begin with the following important remark. 
\begin{remark}
It is not reasonable to expect that the embedding $W^p_\nu(\Omega)\hookrightarrow L^{\phi}(\Omega)$ is compact. This assertion aligns with the well-known fact that in the classical case where $\nu(h)=|h|^{-d-sp}$, the embedding $W^{s,p}(\Omega)\hookrightarrow L^{p_s^*}(\Omega)$ is never compact even for sufficiently smooth $\Omega$. However, the next result devises the compact embedding into Orlicz space $L^{\psi}(\Omega)$. To this end, we point  out that  some conditions on $\nu$ and $\Omega$ under which the embedding $W^p_\nu(\Omega)\hookrightarrow L^{p}(\Omega)$  is compact are referenced for  instance \cite[Section 3.7]{Fog20} and \cite{Fog23s,FK22,DMT23}.
\end{remark}
\begin{theorem}[Compact embedding]\label{thm:compact-emb}
Let the assumptions of Theorem \ref{thm:main-result2} be in force. Let $\Omega\subset\R^d$ be  open and bounded. Assume  the embedding $W^p_\nu(\Omega)\hookrightarrow L^{p}(\Omega)$  is compact. 
Let $\psi $ be a Young function such that, for all $t\geq 1$, $\psi(ct) \leq \max(t, \phi(t))$ $($in particular $\psi(ct) \leq \max(t^p, \phi(t)))$ for some $c>0$. Then the embedding $W^p_\nu(\Omega)\hookrightarrow L^{\psi}(\Omega)$ is compact provided that $\psi$ grows essentially more slowly than $\phi$ at infinity, that is, 
\begin{align*}
\lim_{t\to\infty}\frac{\psi(t)}{\phi(t)}=0. 
\end{align*} 
In particular if $\Omega$ is bounded Lipschitz and $p^*_s>0$,  then  the embedding $W^{s,p}(\Omega)\hookrightarrow L^{q}(\Omega)$ with  $q\in [1, p^*_s)$, is compact. 
\end{theorem}
\begin{proof}
The  compactness of the embedding $W^{s,p}(\Omega)\hookrightarrow L^{q}(\Omega)$ with $q\in [1, p^*_s)$, is a consequence of the first statement since in this case $\phi(t)= \gamma_s^{p^*_s/p}t^{p^*_s}$ and $t^q/t^{p^*_s} \xrightarrow{t\to \infty}0$. To  prove the first statement, let $(u_n)_n\subset \WnuOm$ be bounded sequence in $\WnuOm.$ By Corollary \ref{cor:emb-sobolev-bis}, the Sobolev embedding $W^p_\nu(\Omega)\hookrightarrow L^{\psi}(\Omega)$ is continuous. Moreover, taking $\psi=\phi$, one finds that $(u_n)_n$ is also bounded $L^{\phi}(\Omega)$. By assumption, we can assume without lost of generality that there is $u\in \WnuOm$ such that $\|u_n-u\|_{L^p(\Omega)}\xrightarrow{n\to\infty}0$. For  $\eps>0$, consider $t_\eps>0$ such that $\psi(t)\leq\eps \phi(t)$ for all $t>t_\eps$. Moreover, we know from Proposition \ref{prop:growth-phi} that $t\mapsto \frac{\phi(t)}{t}=  t^{p-1}\frac{\phi(t)}{t^p}$ is increasing. 
Therefore, for all $t\in (0, t_\eps]$ we find that 
\begin{align*}
\frac{\psi(t)}{t} \leq \frac{1}{c}\max\big( 1, \frac{\phi(t/c)}{(t/c)}\big)\leq  \frac{1}{c}\max\big( 1, \frac{\phi(t_\eps/c)}{(t_\eps/c)}\big):= C(\eps). 
\end{align*}
\noindent Whence,  we have just proved that for every $\eps>0$ there is $C(\eps)>0$ such that
\begin{align*}
\psi(t)\leq \eps \phi(t)+ C(\eps)t \qquad\text{for all $t>0$}. 
\end{align*}
From this one readily deduces that, for every $\eps>0$ there is $C(\eps)>0$ 
\begin{align*}
\|u_n-u\|_{L^{\psi}(\Omega)} &\leq \eps \|u_n-u\|_{L^{\phi}(\Omega)} +  C(\eps)\|u_n-u\|_{L^{1}(\Omega)}\leq \eps M 
+  C(\eps) |\Omega|^{1-\frac{1}{p}}\|u_n-u\|_{L^{p}(\Omega)}
\end{align*}
where we put $M= \|u\|_{L^{\phi}(\Omega)}+ \sup_{n\geq1}\|u_n\|_{L^{\phi}(\Omega)}$.  Since $\|u_n-u\|_{L^p(\Omega)}\xrightarrow{n\to\infty}0$, we get 
\begin{align*}
\limsup_{n\to \infty}	\|u_n-u\|_{L^{\psi}(\Omega)} &\leq \eps M. 
\end{align*}
Whence since $\eps>0$ is arbitrarily chosen,  we deduce that   $\|u_n-u\|_{L^{\psi}(\Omega)}\xrightarrow{n\to\infty}0$. 
\end{proof}

\section{Poincar\'e-Sobolev inequality}\label{sec:poincare-sobolev}
In this section, we wish to establish Gagliardo-Nirenberg-Sobolev inequality for functions restricted on a domain $\Omega\subset \R^d$. First and foremost, observe that the inequality
\begin{align*}
\|u\|_{L^\phi(\Omega)} \leq C\Big(\iil_{\Omega\Omega} |u(x)-u(y)|^p\nu(x-y)\d y\d x\Big)^{1/p}\quad\text{for all } \quad u \in L^\phi(\Omega), 
\end{align*}
cannot hold for an arbitrary bounded $\Omega\subset \R^d$. In fact, if u is a nonzero constant, then the right-hand side is zero but the left-hand side is not. Accordingly, we need to rule out constant functions in this context. For instance, if we replace the integrand on the left-hand side by $u -\mbox{$\fint_\Omega u$}$ then it fully makes sense to think of an inequality of the form 
\begin{align*}
\big\|u-\mbox{$\fint_\Omega u$}\big \|_{L^\phi(\Omega)} \leq C\Big(\iil_{\Omega\Omega} |u(x)-u(y)|^p\nu(x-y)\d y\d x\Big)^{1/p}\quad\text{for all } \quad u \in L^\phi(\Omega),
\end{align*}
\noindent where $\fint_\Omega u= \frac{1}{|\Omega|}\int_\Omega u(x)\d x $ denotes the mean value of $u$ over $\Omega$. This particular type of inequality is customarily well known as a Sobolev-Poincar\'e type inequality and turns out to have a strong reciprocity with the Poincare type inequalities. To be more precise, see Theorem \ref{thm:poincare-sobolev-equiv}, the validity of Sobolev-Poincar\'e type inequality implies that of Poincar\'{e} type inequality and vice versa. Let us recall some Poincar\'{e} type inequalities of interest here. Another consequence of Theorem \ref{thm:inequality-indicator} is given by the Poincar\'{e}-Friedrichs type inequality. 

\begin{theorem}[Poincar\'e type inequalities]\label{thm:poincare}
Let $\Omega \subset \R^d$ be measurable with $|\Omega|<\infty$ and let $\nu:\R^d\setminus \{0\}\to[0,\infty)$ be measurable with full support. 

\vspace{2mm}

\noindent \textbf{Poincar\'e-Friedrichs inequality:} Let $L^p_\Omega (\R^d)= \{u\in L^p(\R^d):~u=0, ~\text{a.e on } \Omega^c \}$ and let $\nu^\#$ be defined as in Theorem \ref{thm:inequality-indicator}. Letting $C=[2\nu^\#(|\Omega|)]^{-1/p}$, there holds that 
\begin{align*}
\|u\|_{L^p(\Omega)} \leq C\Big(\iil_{\R^d\R^d} |u(x)-u(y)|^p\nu(x-y)\d y\d x\Big)^{1/p}\quad\text{for all } \quad u \in L^p_\Omega(\R^d) .
\end{align*}
\noindent \textbf{Poincar\'e inequality:} Assume $\nu$ is radially almost decreasing, i.e. $\kappa \nu(|x|)\leq \nu(|y|)$ if $|x|\geq |y|$, then letting $C=[\kappa |\Omega| \nu(R)]^{-1/p} $ where $R= \operatorname{diam}(\Omega)$ is the diameter of $\Omega$, we have 
\begin{align*}
\big\|u-\mbox{$\fint_{\Omega} u$}\big\|_{L^p(\Omega)} \leq C\Big(\iil_{\Omega\Omega} |u(x)-u(y)|^p\nu(x-y)\d y\d x\Big)^{1/p}\quad\text{for all $u \in L^p(\Omega)$}. 
\end{align*}
\end{theorem}

\begin{proof}
If $u=0$ a.e. on $\Omega^c$, Theorem \ref{thm:inequality-indicator} yields the Poincar\'{e}-Friedrichs inequality as follows 
\begin{align*}
|u|^p_{W_\nu^p(\R^d)}
&\geq 2 \int_{\Omega} |u(x)|^p\d x\int_{\Omega^c} \nu(x-y)\d y\geq 2\nu^\#(|\Omega|)\|u\|^p_{L^p(\Omega)}.
\end{align*}
\noindent Next let $R= \operatorname{diam}(\Omega)$ and assume $\nu$ is almost decreasing, so that we get $\nu(x-y)\geq \kappa \nu(R)$ for all $x,y\in \Omega$. Set $C^p=\kappa |\Omega| \nu(R)$. For $u\in L^p(\Omega)$, Jensen's inequality yields, 
\begin{align*}
\iil_{\Omega\Omega} |u(x)-u(y)|^p\nu(x-y)\d y\d x&\geq C^p\int_\Omega \fint_\Omega |u(x)-u(y)|^p\d y\d x \geq C^p\big\|u-\mbox{$\fint_{\Omega} u$}\big\|^p_{L^p(\Omega)}. 
\end{align*}
\end{proof}

\noindent More recent and improved versions of Poincar\'{e} type inequalities including the situation where $\nu$ is not fully supported and/or  $\Omega$ is bounded only in one direction have been established in \cite{Fog23s}. In particular, one is able to obtain the Poincar\'{e} inequality even  if $\nu$ is not almost decreasing, by requiring that  the embedding $W_\nu^p(\Omega) \hookrightarrow L^p(\Omega)$ is compact; see  \cite{Fog20,Fog23s}. The next corollary follows from Theorem \ref{thm:main-result2}.

\begin{corollary}[Poincar\'e-Friedrichs-Sobolev type inequality]
\noindent Let $\Omega\subset \R^d$ be open and define $L^\phi_\Omega (\R^d)= \{u\in L^\phi(\R^d):~u=0, ~\text{a.e on } \Omega^c \}$. Under the assumptions of Theorem \ref{thm:main-result2} we get
\begin{align*}
\|u\|_{L^\phi(\Omega)} \leq \Theta_t\Big(\iil_{\R^d\R^d} |u(x)-u(y)|^p\nu(x-y)\d y\d x\Big)^{1/p}\quad\text{for all $u \in L^\phi_\Omega(\R^d)$}.
\end{align*}
\end{corollary}

\noindent The next result is somewhat a side consequence of the convexity assumption on $\phi_p$ and shows the equivalence between the Sobolev inequality and the Poincar\'{e}-Sobolev inequality. 

\begin{theorem}[Poincar\'e-Sobolev inequalities]\label{thm:poincare-sobolev-equiv} Assume assumptions of Theorem \ref{thm:main-result2} hold. Let $\Omega \subset \R^d$ be measurable such that $0<|\Omega|<\infty$. If the Poincar\'{e}-Sobolev inequality holds true $\Omega$, i.e., 
\begin{align*}
\big\|u-\mbox{$\fint_{\Omega} u$}\big\|_{L^\phi(\Omega)} \leq C \Big(\iil_{\Omega\Omega} |u(x)-u(y)|^p\nu(x-y)\d y\d x\Big)^{1/p}\quad\text{for all $u \in L^\phi(\Omega)$},
\end{align*}
then the Poincar\'{e} inequality also holds true, i.e., 
\begin{align*}
\big\|u-\mbox{$\fint_{\Omega} u$}\big\|_{L^p(\Omega)} \leq C \Big(\iil_{\Omega\Omega} |u(x)-u(y)|^p\nu(x-y)\d y\d x\Big)^{1/p}\quad\text{for all $u \in L^p(\Omega)$}. 
\end{align*}
The converse holds true if in addition, $\Omega$ is an $W^p_\nu$-extension domain. 
\end{theorem}

\begin{proof}
By Proposition \ref{prop:growth-phi} we get $\phi(t)\geq \delta_0 t^p$ for all $t\geq t_0$ with fixed $t_0>0$. Thus, the embedding $L^\phi(\Omega) \hookrightarrow L^p(\Omega)$ is continuous, by Theorem \ref{thm:orlicz-emb}. Together with the Poincar\'{e} inequality yields  
\begin{align*}
\big\|u-\mbox{$\fint_{\Omega} u$}\big\|_{L^p(\Omega)} \leq C\big\|u-\mbox{$\fint_{\Omega} u$}\big\|_{L^\phi(\Omega)} \leq C\Big(\iil_{\Omega\Omega} |u(x)-u(y)|^p\nu(x-y)\d y\d x\Big)^{1/p}.
\end{align*}

\noindent Conversely assume, the Poincar\'{e} inequality holds and $\Omega$ is an $W^p_\nu$-extension domain. Let $\overline{u} \in W^p_\nu(\R^d)$ be an extension of $u_0= u-\fint_{\Omega} u$ with $u\in W^p_\nu(\Omega)$. Applying Theorem \ref{thm:main-result2} reveals that $\overline{u} \in L^\phi(\R^d)$ and we deduce the Poincar\'{e}-Sobolev inequality as follows
\begin{align*}
\big\|u-\mbox{$\fint_{\Omega} u$}\big\|_{L^\phi(\Omega)} 
\leq \|\overline{u} \|_{L^\phi(\R^d)} &\leq C\Big(\iil_{\R^d\R^d } |\overline{u} (x)-\overline{u} (y)|^p\nu(x-y)\d y\d x\Big)^{1/p}\\
&\leq C\Big(\int_{\Omega} |u(x)-\mbox{$\fint_{\Omega} u$}|^p\d x+\iil_{\Omega\Omega} |u_0(x)-u_0(y)|^p\nu(x-y)\d y\d x \Big)^{1/p}\\
&\leq C\Big(\iil_{\Omega\Omega} |u(x)-u(y)|^p\nu(x-y)\d y\d x \Big)^{1/p}.
\end{align*}
\end{proof}

\noindent As a direct consequence of Theorem \ref{thm:main-result2} and Theorem \ref{thm:poincare-sobolev-equiv} combined with Theorem \ref{thm:poincare} we get.

\begin{corollary} Let the assumptions of Theorem \ref{thm:main-result2} be in force. If $\nu:\R^d\setminus\{0\} \to [0,\infty)$ is almost decreasing and $\Omega\subset \R^d$ is an $W^p_\nu$-extension domain then Poincar\'{e}-Sobolev inequality holds, that is, there is a constant $C= C(\Omega, \nu, p, d)>0$ only depends on $\Omega, \nu, d$ and $p$ such that 
\begin{align}\label{eq:sobolev-poincare}
\qquad\big\|u-\mbox{$\fint_{\Omega} u$}\big\|_{L^\phi(\Omega)} \leq C \Big(\iil_{\Omega\Omega} |u(x)-u(y)|^p\nu(x-y)\d y\d x\Big)^{1/p}\quad\text{for all }\, \, u \in L^\phi(\Omega).
\end{align}
\noindent In particular, if $s\in (0,1)$ and $p^*_s>0$ then 
\begin{align*}
\big\|u-\mbox{$\fint_{\Omega} u$}\big\|_{L^{p^*_s}(\Omega)} \leq C\Big( \iil_{\Omega\Omega} \frac{|u(x)-u(y)|^p}{|x-y|^{d+sp}} \d y\d x\Big)^{1/p}\quad\text{for all }\, \, u\in L^{p^*_s}(\Omega).
\end{align*}
\end{corollary}
	
\begin{remark}
Let $\psi $ be a Young function such that, for all $t\geq 1$, $\psi(ct) \leq \max(t, \phi(t))$ for some $c>0$. Assume that the Sobolev-Poincar\'e inequality \eqref{eq:sobolev-poincare} holds then  there is $C'>0$ such that 
\begin{align*}
\qquad\big\|u-\mbox{$\fint_{\Omega} u$}\big\|_{L^\psi(\Omega)} \leq C \Big(\iil_{\Omega\Omega} |u(x)-u(y)|^p\nu(x-y)\d y\d x\Big)^{1/p}\quad\text{for all }\, \, u \in L^\psi(\Omega).
\end{align*}
\end{remark}

\begin{remark}
\noindent By a straightforward scaling argument one finds 
$C= C(d,p, s)$ such that, 
\begin{align}\label{eq:fractional-unif}
\big\|u-\mbox{$\fint_{B_r} u$}\big\|_{L^{p^*_s}(B_r)} \leq C\Big( \iil_{B_rB_r} \frac{|u(x)-u(y)|^p}{|x-y|^{d+sp}} \d y\d x\Big)^{1/p}\quad\text{for all $u\in L^{p^*_s}(B_r)$},
\end{align}
\noindent where $B_r= B(0, r)$, $r>0$ is any ball and $C>0$ in \eqref{eq:fractional-unif} is independent on $r>0$. Letting $r\to \infty$, one verifies  that the uniform estimate in \eqref{eq:fractional-unif} implies the Gagliardo-Nirenberg-Sobolev inequality given in Theorem \ref{thm:gagliardo-sobolev}. More generally the next result infers that the uniform Poincar\'{e}-Sobolev inequality on balls with respect to $\phi$ implies the nonlocal Gagliardo-Nirenberg-Sobolev inequality. 
\end{remark} 
\begin{theorem}\label{thm:poincare-sobolev}
Let $\nu \in L^1(\R^d, 1\land |h|^p)$ be nonnegative and let $\phi(t)= 1/w^{-1}(1/t)$ be defined as in Theorem \ref{thm:main-result2}, with $w(r)= (r\nu^{\#}(r))^{1/p}$. Assume there is a universal constant $\Theta>0$ such that for all balls $B\subset \R^d$, the following Poincar\'{e}-Sobolev inequality holds true,
\begin{align*}
\big\|u-\mbox{$\fint_{B} u$}\big\|_{L^\phi(B)} \leq \Theta \Big(\iil_{BB} |u(x)-u(y)|^p\nu(x-y)\d y\d x\Big)^{1/p}\quad\text{for all $u \in L^\phi(B)$}.
\end{align*}
\noindent Then following inequality holds true as well
\begin{align*}
\|u\|_{L^\phi(\R^d)} \leq \Theta \Big(\iil_{\R^d\R^d} |u(x)-u(y)|^p\nu(x-y)\d y\d x\Big)^{1/p}\quad\text{for all $u \in L^\phi(\R^d)$}.
\end{align*}
\end{theorem}

\begin{proof}
Let $B_r= B(0,r)$, $r>0$. Using the assumption and the formula  assumption \eqref{eq:luxem-indicator}  gives 
\begin{align*}
\|u\|_{L^\phi(B_r)}
&\leq \big\|u-\mbox{$\fint_{B_r} u$}\big\|_{L^\phi(B_r)}
+\|\mathds{1}_{B_r}\|_{L^\phi(\R^d)}\Big|\fint_{B_r}u\Big|\\
&\leq \Theta \Big(\iil_{\R^d\R^d } |u(x)-u(y)|^p\nu(x-y)\d y\d x\Big)^{1/p}+ \frac{|B_r|^{-1/p}}{\phi^{-1}( 1/|B_r|)} \Big(\int_{\R^d}|u(x)|^p\d x\Big)^{1/p}.
\end{align*}

\noindent In virtue of Theorem \ref{thm:inequality-indicator}, we know that $\nu^{\#}(r)
\to 0$ as $r\to \infty$ so, by the definition of $\phi$ we obtain
\begin{align*}
\frac{|B_r|^{-1/p}}{\phi^{-1}( 1/|B_r|)}= |B_r|^{-1/p}w(|B_r|) =(\nu^{\#}(r))^{1/p}
\to 0 \quad \text{as}\quad r\to \infty. 
\end{align*}
Since $\Theta>0$ is independent of $r$, tending $r\to\infty$ in the foregoing yields the desired inequality. 
\end{proof}
\textbf{Open Questions:} 
$(i)$ Regarding Theorem \ref{thm:main-partial}, can the growth condition \eqref{eq:growth-condition} be improved? ($ii)$ Is the critical function $\phi$ optimal? Indeed if $\psi$ is another Young function satisfying the Gagliardo-Nirenberg-Sobolev inequality \eqref{eq:main-equality} then it also holds true for the Young function $t\mapsto \max(\phi(t), \psi(t))$. Note that we call $\phi$  optimal if there exists $c>0$ such that $\psi(t)\leq \phi(ct)$ for all $t>0$, this is equivalent  
(see Theorem \ref{thm:orlicz-emb}) to saying that $L^{\phi} (\R^d) \hookrightarrow L^{\psi}(\R^d)$; in other words $L^{\phi} (\R^d)$ is the smallest Orlicz space satisfying the inequality.  \eqref{eq:main-equality}.   

\bibliographystyle{alpha}

\end{document}